\documentclass[journal,twoside,web]{ieeecolor}
\usepackage{generic}
\usepackage{cite}
\usepackage{amsmath,amssymb,amsfonts}
\usepackage{algorithmic}
\usepackage{graphicx}
\usepackage{algorithm,algorithmic}
\usepackage{hyperref}
\hypersetup{hidelinks=true}
\usepackage{textcomp}

\newtheorem{lemma}{Lemma}
\newtheorem{remark}{Remark}
\newtheorem{theorem}{Theorem}
\newtheorem{definition}{Definition}
\newtheorem{corollary}{Corollary}%[proposition]
\newtheorem{assumption}{Assumption}%[section]
\newtheorem{example}{Example}%[section]
\newcommand{\argmin}{\operatornamewithlimits{arg\,min}}          % argmin

\def\BibTeX{{\rm B\kern-.05em{\sc i\kern-.025em b}\kern-.08em
    T\kern-.1667em\lower.7ex\hbox{E}\kern-.125emX}}
\begin{document}
\title{Dual Effect and Certainty Equivalence Revisited \\ for a Linear-Gaussian Dual Control Problem}
% \author{First A. Author, \IEEEmembership{Fellow, IEEE}, Second B. Author, and Third C. Author Jr., \IEEEmembership{Member, IEEE}
% \thanks{This paragraph of the first footnote will contain the date on 
% which you submitted your paper for review. It will also contain support 
% information, including sponsor and financial support acknowledgment. For 
% example, ``This work was supported in part by the U.S. Department of 
% Commerce under Grant 123456.'' }
% \thanks{The next few paragraphs should contain 
% the authors' current affiliations, including current address and e-mail. For 
% example, First A. Author is with the National Institute of Standards and 
% Technology, Boulder, CO 80305 USA (e-mail: author@boulder.nist.gov). }
% \thanks{Second B. Author Jr. was with Rice University, Houston, TX 77005 USA. He is 
% now with the Department of Physics, Colorado State University, Fort Collins, 
% CO 80523 USA (e-mail: author@lamar.colostate.edu).}
% \thanks{Third C. Author is with 
% the Electrical Engineering Department, University of Colorado, Boulder, CO 
% 80309 USA, on leave from the National Research Institute for Metals, 
% Tsukuba, Japan (e-mail: author@nrim.go.jp).}}
\author{Ying Wang, Yue Ju and H{\aa}kan Hjalmarsson,  \IEEEmembership{Fellow, IEEE}
\thanks{This work was supported by the VINNOVA Competence Center AdBIOPRO 2016-05181, the Swedish Research Council through project 2025-04833; and the
Wallenberg AI, Autonomous Systems and Software Program (WASP),
funded by the Knut and Alice Wallenberg Foundation.}%
\thanks{Ying Wang, Yue Ju and H{\aa}kan Hjalmarsson are with the
Department of Decision and Control Systems, KTH Royal Institute
of Technology, SE-100 44 Stockholm, Sweden
(e-mail: yinwang@kth.se; yuej@kth.se; hjalmars@kth.se).}%
\thanks{H{\aa}kan Hjalmarsson is also with the Center for
Advanced Bio Production AdBIOPRO, KTH Royal Institute of
Technology, SE-100 44 Stockholm, Sweden.}}

\maketitle

\begin{abstract}
This paper revisits the relation between dual effect and certainty equivalence (CE) through a scalar linear-Gaussian model. 
In dual control, the current input affects not only the immediate cost but also the information available for future decisions, and hence the future cost. 
Dynamic programming (DP) provides a natural framework for formulating dual control problems. 
At each step of the Bellman recursion, an optimal input minimizes
the sum of the stage cost and the expected cost-to-go.
We restrict attention to stage costs minimized by the CE input.
However, CE may not minimize the Bellman objective because of
the input-dependent expected cost-to-go.
Classical no-dual-effect (NDE) conditions guarantee CE optimality in certain problem classes by requiring the conditional distribution of the estimation error to be independent of past controls.
In the present model, we observe that CE can remain optimal at the level of a single Bellman recursion even when the classical NDE conditions fail, 
provided that the expected cost-to-go is independent of the current input.
Motivated by this observation, we introduce a notion of NDE formulated directly at the level of  the Bellman objective, which we call Bellman-side NDE.
Under some mild regularity conditions, we show that Bellman-side NDE
%for the scalar linear-Gaussian model
can be characterized by a heat equation.
This characterization identifies the
class of Bellman-side NDE functions, and further yields a decomposition of a generic cost-to-go function into a Bellman-side NDE component and a dual-effect component.
\end{abstract}

\begin{IEEEkeywords}
Certainty Equivalence, Dual Control, Adaptive Control, Dynamic Programming.
% Enter key words or phrases in alphabetical order, separated by commas. Using the IEEE Thesaurus can help you find the best standardized keywords to fit your article. Use the thesaurus access request form for free access to the IEEE Thesaurus: \underline{https://www.ieee.org/publications/services/thesaurus-acce}\\
% \underline{ss-page.com.}
\end{IEEEkeywords}

\section{Introduction}

Certainty equivalence (CE) is the simplest and most widely used policy in partially observed stochastic control problems~\cite{astrom1995adaptive}. 
CE replaces the unknown parameter by its current estimate and acts as if this estimate were correct. 
Remarkably, in the partially observed linear-quadratic-Gaussian (LQG) problem, this simple policy is  optimal~\cite{joseph1961linear,gunckel1963general}.
Beyond LQG, CE is generally suboptimal, since it ignores the value of information that the input generates for future decisions.
Despite the simplification CE offers, a transparent characterization of \emph{when} it is optimal is still lacking.
This question is the subject of the present paper.

% By contrast, CE minimizes the current cost based on the current parameter estimate and thus serves as the exploitation input. 
% Whether CE also minimizes the Bellman objective depends on the input dependence of the expected cost-to-go: CE is optimal if and only if the CE input remains a minimizer after this future-cost term is included.

The classical answer to the CE-optimality question is the no-dual-effect (NDE) condition of Bar-Shalom and Tse~\cite{barshalom1974dual}: under the model and cost assumptions considered there, CE is optimal if the conditional distribution of the estimation error is independent of past controls. We will refer to this as posterior-side NDE.
Moreover, NDE was argued to be a necessary and sufficient condition for CE optimality in~\cite{barshalom1974dual}.
Derpich and Y\"uksel~\cite{derpich2023dual} identified a subtle error in Bar-Shalom and Tse's
argument that CE optimality implies NDE, and gave a counterexample. They introduced a relaxed posterior-side NDE notion, requiring the above independence over a restricted policy class rather than all admissible policies. This weaker condition still suffices for CE optimality in the problem class considered there.
Recent work on dual model predictive control~\cite{baltussen2026separation} adopts a
related perspective, proposing metrics that empirically quantify how the policy depends on the posterior covariance.

These posterior-side conditions concern the dependence of the estimation-error distribution on past controls.
Dynamic programming (DP) offers a complementary perspective: at each stage, the optimal input minimizes the Bellman objective, which combines the stage cost with the expected cost-to-go.
This structure exposes the input's dual effect~\cite{feldbaum1960dual}:
the input determines the current cost and can also affect the expected cost-to-go through the information it generates for future decisions.
Balancing current performance against the value of this information gives rise to the exploration--exploitation tradeoff at the core of dual control~\cite{feldbaum1960dual,astrom1995adaptive,mesbah2018stochastic}.

Building on this DP perspective, we address the question of
CE optimality by examining the input dependence of the expected
cost-to-go at the level of a single Bellman recursion.
When the CE input minimizes the stage cost, it also minimizes
the Bellman objective if the expected cost-to-go is independent
of the input.
This independence can hold even when the conditional distribution
of the next-step estimation error depends on the current input, i.e., when classical posterior-side NDE fails.
This observation motivates a Bellman-side NDE condition formulated directly in terms of the expected cost-to-go.

This perspective is natural, but exploiting it requires a benchmark model:
one that preserves the essential dual effect while keeping the state dynamics
explicit enough to expose how the input acts on the expected cost-to-go.
Existing settings entangle several effects at once, such as system dynamics,
nonlinear posterior updates, and exploration directions in multiple-input
problems. These make it difficult to isolate what exactly makes CE optimal or
not.

We therefore study the scalar linear-Gaussian model 
\begin{equation}
    y_t=\theta u_t+e_t,\label{eq:io}
\end{equation}
where $\theta$ is an unknown parameter, $u_t$ is the input, $y_t$ is the scalar observation, and $e_t$ is Gaussian noise. Under a Gaussian prior on $\theta$, the posterior distribution remains Gaussian with mean $m_t$ and variance $P_t$, which together form the belief state $x_t=(m_t,P_t)$.
This static and scalar model strips away system dynamics and multidimensional exploration while retaining the dual effect. Even here, the backward Bellman recursion generally does not yield closed-form expressions for the optimal cost-to-go functions.
The literature mostly reports approximate, numerical, or asymptotic solutions~\cite{prescott1972multi,taylor1974asymptotic,lobo1999policies,wieland2000learning,wieland2000monetary,han2006approximate,wang2024regret,wang2025revisiting}. Hence, our analysis does not attempt to solve this recursion or to establish CE
optimality at every stage. Instead, we characterize CE optimality at the level
of a single Bellman recursion for a generic cost-to-go function, which needs not
arise from the Bellman recursion of a specific control problem.
%Carrying this out requires a benchmark model in which the Bellman objective can be analyzed exactly. 
%We therefore study the scalar linear-Gaussian model $y_t=\theta u_t+e_t$ with a Gaussian prior~\cite{prescott1972multi,wang2025revisiting}. This model provides a minimal but nontrivial benchmark. It strips away the interference of system dynamics and multidimensional exploration directions while retaining the essential dual effect. It admits closed-form dynamics for the posterior mean $m$ and variance $P$, and the variance depends explicitly on the input, so posterior-side NDE fails. 
This viewpoint allows us to identify a nontrivial class of generic cost-to-go functions for which CE is optimal for the present model. The contributions are as follows.
\begin{itemize}
\item We introduce a Bellman-side notion of NDE: a cost-to-go function $J$ has no dual effect at a belief state if the expected cost-to-go it induces is insensitive to the current input. 
This NDE notion yields a sufficient condition for CE optimality.
\item We characterize Bellman-side NDE analytically: 
% under mild regularity, 
$J$ is Bellman-side NDE if and only if 
% $\partial_P J-\frac{1}{2}\partial_{mm} J=0$, i.e., 
$J$ solves a heat equation. This NDE notion thus admits a local differential characterization.
\item We prove that Bellman-side NDE functions are exactly the Gaussian convolutions of their complete-information boundary functions within a certain function class.
%The heat polynomials span the subclass generated by polynomial boundary functions, which is a subclass of all Bellman-side NDE functions.
%, with the heat polynomials being the basis of the polynomial function space.
\item We decompose a generic cost-to-go function into a Bellman-side NDE component and a dual-effect component, and show that the entire exploration incentive is carried by the latter.
\end{itemize}

The remainder of the paper is organized as follows. Section~\ref{sec:problem} formulates the scalar linear-Gaussian dual control problem and its belief-state DP formulation. Section~\ref{sec:definition} introduces Bellman-side NDE and compares it with posterior-side NDE. Section~\ref{sec:source} establishes the source characterization of Bellman-side NDE. Section~\ref{sec:function_space} develops the Gaussian convolution representation of Bellman-side NDE functions and characterizes the corresponding NDE function space. Section~\ref{sec:bellman_side_decomposition} derives the cost-to-go function decomposition. Section~\ref{sec:numerical} presents the numerical studies, and Section~\ref{sec:conclusion} concludes the paper.

\paragraph*{Notation}
The sets of real and positive real numbers are denoted by $\Re$ and $\Re_{++}$. The notation $E\{X\}$ denotes the expectation of $X$.
% , while $E_\theta\{X\}$ denotes the expectation under a fixed parameter value $\theta$, with the expectation taken over all remaining randomness.
The notation $X\sim N(a,b)$ means that
$X$ is Gaussian with mean $a$ and variance $b$.
The notation $X\mid Y$ denotes the
conditional distribution of $X$ given $Y$. For $c\in\Re$, $\operatorname{sgn}(c)\in\{-1,+1\}$ denotes the sign of $c$, with
the convention $\operatorname{sgn}(0)=1$.
The notation \(x\downarrow0\) denotes the limit as \(x\) approaches
zero from above.
For functions $a,b$ of a real variable, $a(x) = o(b(x))$ as $x \downarrow 0$ 
means $a(x)/b(x) \to 0$ as $x \downarrow 0$. 
The notation $C^\infty(X)$ denotes the set of functions on $X$ 
possessing continuous partial derivatives of all orders.

\section{A Scalar Linear-Gaussian Dual Control Problem}
\label{sec:problem}

Section~\ref{subsec:dual_control_formulation} formulates the scalar linear-Gaussian dual control problem in terms of the available information.
The challenge here is that information grows with time.
The remainder of the section therefore recasts the problem as a DP with a two-dimensional belief state.
Specifically, Section~\ref{subsec:belief_state_dynamics} derives the belief state dynamics.
Section~\ref{subsec:belief_state_problem} formulates this belief-state DP, and Section~\ref{subsec:risk_to_state} specifies the stage cost.  Finally, Section~\ref{subsec:research_problem} states the research problem: when does CE optimality hold?

\subsection{Information-Based Dual Control Formulation}
\label{subsec:dual_control_formulation}

We begin from the known-parameter decision problem. Suppose the
parameter $\theta\in\Re$ in~\eqref{eq:io} is known. The controller then chooses an
input $u\in U=\Re$ by minimizing a scalar loss function $l(\theta,u)$. 
We define $\upsilon^*(\theta)$, which we will refer to as an \emph{oracle input}, to be a minimizer of the loss:
\begin{equation}\label{eq:oracle}
	\upsilon^*(\theta)
	\in
	\argmin_{u\in U} l(\theta,u).
\end{equation}
% denote an \emph{oracle input}.

In dual control, the parameter $\theta$ is unknown and must be learned from input-output measurements generated by the scalar linear-Gaussian model~\eqref{eq:io},
% \begin{equation}\label{eq:io}
%     y_t=\theta u_t+e_t,
%     \qquad t=0,1,\ldots,
% \end{equation}
where $\{e_t\}$ is an independent and identically distributed (i.i.d.) sequence of standard Gaussian random
variables. %The prior uncertainty of $\theta$ is modeled as
We further assume that $\theta$ is a random variable with distribution
\[
    \theta\sim N(m_0,P_0),
    \qquad
    m_0\in\Re,\quad P_0\in\Re_{++},
\]
with $\theta$ being independent of the noise sequence $\{e_t\}$.

We formulate the dual control problem in terms of the
\emph{information} available to the controller. Let $I_0=(m_0,P_0)$
denote the initial information, and for $t\geq1$ define the
information available before choosing $u_t$ by
% \footnote{Since each
% input $u_s=\bar\mu_s(I_s)$, $s=0,1,..., t-1$ is a deterministic function of the past observations, the information $I_t$ is generated by the initial information $I_0$ and the observations
% $(y_0,\ldots,y_{t-1})$.}
\[
    I_t
    =
    (m_0,P_0,u_0,y_0,\ldots,u_{t-1},y_{t-1}).
\]
After the input $u_t$ is applied and the measurement $y_t$ is
obtained, the information evolves according to
$I_{t+1}=(I_t,u_t,y_t)$.

The inputs over a horizon $T$ are determined by a causal
\emph{policy} $\bar\pi_T=\{\bar\mu_0,\ldots,\bar\mu_{T-1}\}$, where each $\bar\mu_t$ is an $I_t$-measurable decision rule resulting in the input at time $t$ being
\[
    u_t=\bar\mu_t(I_t),
    \qquad
    t=0,\ldots,T-1.
\]
Here the input $u_t$ has a dual effect~\cite{feldbaum1960dual}: it incurs
the current loss $l(\theta,u_t)$, and it determines how much
information the next measurement $y_t$ carries about $\theta$.
% through the model~\eqref{eq:io}.

For a policy $\bar\pi_T\in\bar\Pi_T$, with $\bar\Pi_T$ denoting the class of causal
policies, define the cost function
\begin{equation}
        J_{\bar\pi_T}(I_0)
    =
    E\Big\{
        \sum_{t=0}^{T-1}
        l\big(\theta,\bar\mu_t(I_t)\big)
    \Big\},\label{eq:I_cost}
\end{equation}
where the expectation is over the prior distribution of $\theta$ and
the noise sequence. The finite-horizon dual control problem is
\begin{equation}
        \bar\pi_T^*
    \in
    \argmin_{\bar\pi_T\in\bar\Pi_T}
    J_{\bar\pi_T}(I_0). \label{eq:I_Dual}
\end{equation}

In this information-based formulation, the dimension of $I_t$ grows with $t$, so the decision rules $\bar\mu_t$ act on domains of increasing dimension.
Optimizing over such a policy class directly is intractable.
To handle this, we proceed by recasting the problem as a DP with a two-dimensional state.

%In this information-based formulation, the dimension of $I_t$ grows with $t$.
%Consequently, using $I_t$ directly as the state leads to an increasingly high-dimensional state space, which is computationally intractable. This motivates us to replace $I_t$ with an equivalent fixed-dimensional sufficient state

\subsection{Belief-State Dynamics}
\label{subsec:belief_state_dynamics}

Under the Gaussian prior and the linear-Gaussian
model~\eqref{eq:io}, the posterior distribution of $\theta$ remains Gaussian at every stage~\cite[Sec.11]{ljung1999system}, i.e.,
\begin{equation}\label{eq:Gaussianity}
 \theta\mid I_t\sim N(m_t,P_t),
\end{equation}
where $m_t$ and $P_t$ denote the posterior mean and variance, respectively. 
We then define the two-dimensional \emph{belief state}
\[
x_t=(m_t,P_t)\in X,
\qquad
X=\Re\times\Re_{++},
\]
which fully parametrizes the conditional distribution of $\theta$
given $I_t$, and thus serves as a sufficient statistic for $\theta$~\cite{Lehmann98}.
The mean and variance update recursively~\cite[Sec.11]{ljung1999system}:
\begin{align}
    m_{t+1}
    &=
    m_t
    +
    \frac{P_tu_t}{1+P_tu_t^2}
    \bigl(y_t-u_tm_t\bigr),
    \label{eq:theta_update_app}\\
    P_{t+1}
    &=
    \frac{P_t}{1+P_tu_t^2}.
    \label{eq:P_update_app}
\end{align}

The update~\eqref{eq:theta_update_app}-\eqref{eq:P_update_app} is more
convenient in a form that separates the effect of the input from the randomness
of the measurement. The next lemma provides such a form, in which the input
enters only via the \emph{variance reduction} it induces,
\begin{equation}\label{eq:eta}
	\eta(P_t,u_t) = P_t - P_{t+1} = \frac{P_t^2u_t^2}{1+P_tu_t^2},
\end{equation}
while the randomness enters through a standardized innovation,
\begin{equation}\label{eq:w}
	w_t = \mathrm{sgn}(u_t)\,\frac{y_t-u_tm_t}{\sqrt{1+P_tu_t^2}}.
\end{equation}

\begin{lemma}[Belief-state transition]
\label{lem:belief_state_transition_eta}
Consider the measurement model~\eqref{eq:io} with belief state
$x_t=(m_t,P_t)$ and $I_t$-measurable input $u_t$. Then the standardized innovation satisfies $w_t\mid (x_t,u_t)\sim N(0,1)$, and the update \eqref{eq:theta_update_app}-\eqref{eq:P_update_app} can be expressed as
\begin{equation}
    m_{t+1}=
        m_t+\sqrt{\eta(P_t,u_t)}\,w_t,\quad
    P_{t+1}=P_t-\eta(P_t,u_t),
    \label{eq:dynamics_m_P_def}
\end{equation}
where $\eta(P_t,u_t)$ is given in~\eqref{eq:eta}.
\end{lemma}

\begin{proof}
See Appendix~\ref{app:belief_state_transition_eta}.
\end{proof}

With the above lemma, the belief-state transition in~\eqref{eq:dynamics_m_P_def} admits the compact representation
\begin{equation}
    x_{t+1}
    =
    f(x_t,u_t,w_t)
    =
    \bigl(
        m_t+\sqrt{\eta(P_t,u_t)}\,w_t,\;
        P_t-\eta(P_t,u_t)
    \bigr).
    \label{eq:dynamics_f_def}
\end{equation}
In the subsequent analysis, we suppress the arguments of $\eta$ when there is no
confusion. For a fixed $P_t>0$, the map $u_t\mapsto\eta(P_t,u_t)$ sends $\Re$
onto $[0,P_t)$. Solving $\eta_t=\eta(P_t,u_t)$ gives
\[
    u_t=\pm\sqrt{\frac{\eta_t}{P_t(P_t-\eta_t)}},
\]
so every admissible variance reduction $\eta_t\in[0,P_t)$ can be generated by some finite input.
% unique up to sign for $\eta>0$. Note that $\eta\to P$
% requires $|u|\to\infty$: the posterior variance cannot be annihilated
% in one step.

\subsection{Belief-State DP Formulation}
\label{subsec:belief_state_problem}

To reformulate the dual control problem into a belief-state DP, we use a standard result in partially observed stochastic control:  a partially observed problem can be reformulated as a fully observed belief-state problem by taking the posterior distribution of the unobserved state as the state of the DP; see~\cite[Ch.~10]{bertsekas1978stochastic}, \cite[Ch.~6]{kumar1986stochastic}, and \cite[Ch.~8]{Krishnamurthy2025Partially}.
In the present problem, the unobserved state is the static parameter
$\theta$. Its posterior remains Gaussian for each $t$, stated in~\eqref{eq:Gaussianity}.
Hence, the posterior distribution of $\theta$ is exactly parameterized by the two-dimensional
belief state
$
x_t=(m_t,P_t)
$
introduced in Section~\ref{subsec:belief_state_dynamics}.
The growing information $I_t$ can thus be summarized by $x_t$ for the purpose of DP.

The corresponding \emph{belief-state DP} has state $x_t$ and state transition $f$ given by Lemma~\ref{lem:belief_state_transition_eta}. We now specify the stage cost as the \emph{posterior risk}, obtained by taking the conditional expectation of the original loss under the current posterior. Specifically, for $x_t=(m_t,P_t)\in X$, define the \emph{stage cost}:
\begin{equation}\label{eq:stage_cost}
    g(x_t,u_t)
    =
    E
    \big\{
        l(\theta,u_t)
    \big\},
\end{equation}
where the expectation is taken over $\theta\mid x_t\sim N(m_t,P_t)$.
% Thus, $g(x_t,u_t)$ is the posterior risk of applying $u_t$ when the posterior distribution of $\theta$ is represented by $x_t$.

We now recast the information-based dual control problem of
Section~\ref{subsec:dual_control_formulation} in belief-state DP form.
We are interested in a belief-state feedback policy over a horizon $T$: $\pi_T=\{\mu_0,\ldots,\mu_{T-1}\}$ with measurable decision rule $\mu_t:X\to U$, and the policy set is denoted
$\Pi_T$. The cost function of $\pi_T$ from $x_0=(m_0,P_0)$ is
\begin{equation}
\label{eq:cost_f_2i}
    J_{\pi_T}(x_0)
    =
    E\Big\{
        \sum_{t=0}^{T-1}
        g\big(x_t,\mu_t(x_t)\big)
        +
        J_T(x_T)
    \Big\},
\end{equation}
where $x_t$ evolves according to~\eqref{eq:dynamics_f_def}, the
expectation is over 
%the prior $x_0$ and 
the standardized innovation $\{w_t\}$, and
$J_T:X\to\Re$ is a terminal cost. 
Here we choose \(J_T\equiv0\), so that the belief-state cost~\eqref{eq:cost_f_2i} coincides with the cost~\eqref{eq:I_cost} in the information-based formulation.
% .
% : by the cost sufficiency of
% Lemma~\ref{lem:info-state}(a),
% \[
%     E\Big\{
%         \sum_{t=0}^{T-1}
%         c\big(\theta,u_t\big)
%     \Big\}
%     =
%     J_{\pi_T}(x_0),
% \]
The original dual control problem in~\eqref{eq:I_Dual} thus reduces to the belief-state DP problem
\begin{equation}
\label{eq:belief_problem}
    J_0^*(x_0)
    =
    \inf_{\pi_T\in\Pi_T}
    J_{\pi_T}(x_0).
\end{equation}
% The general terminal cost $J_T$ is carried along for the
% generality of the recursion below; it also accommodates approximate
% cost-to-go functions, which reappear in
% Section~\ref{subsec:research_problem}.

The optimal cost~\eqref{eq:belief_problem} is computed by DP: the backward recursion is initialized at $J_T^*=J_T$
and, for $t=T-1,\ldots,0$,
\begin{equation}
\label{eq:finite_horizon_bellman}
    J_t^*(x_t)
    \!=\!
    \inf_{u_t\in U}
    \left[
        g(x_t,u_t)
        \!+\!
        E\big\{
            J_{t+1}^*\big(f(x_t,u_t,w_t)\big)
        \big\}
    \right],
    \
    \forall x_t\!\in\! X.
\end{equation}
Every stage of this recursion is a minimization of the same \emph{Bellman objective} form: a stage cost plus an \emph{expected cost-to-go}. 
Since this form does not change with $t$, we suppress the time subscript and write $x=(m,P)\in X$ and $u\in U$ for a belief state and input, in place of $x_t=(m_t,P_t)$ and $u_t$, respectively.

\subsection{CE-Optimal Stage Cost}
\label{subsec:risk_to_state}

Recall the definition of the oracle input $\upsilon^*(\theta)$ in~\eqref{eq:oracle}. The CE input replaces the unknown
$\theta$ by its posterior mean,
\begin{equation}\label{eq:ce_input}
	\nu(x)=\upsilon^*(m),\qquad x\in X.
\end{equation}

The stage cost $g(x,u)$ is defined in~\eqref{eq:stage_cost} as the posterior risk of the known-parameter cost $l$.
We restrict attention to those $l$ for which the CE input already minimizes
$g(x,\cdot)$, so that CE coincides with the one-step greedy policy.
% The dual-effect mechanism indicates that any failure of CE optimality originates from the dependence of the expected cost-to-go on the current input.
The class of stage costs we consider is specified as follows.

\begin{assumption}[CE-optimal stage cost]
\label{ass:ce_centered_cost}
The stage cost $g(x,u)$, defined in~\eqref{eq:stage_cost}, admits the decomposition
\begin{equation}
    g(x,u)=\ell\bigl(x,u-\nu(x)\bigr)+\psi(x), \qquad x\in X, \ u\in U,
    \label{eq:cost}
\end{equation}
where the map $r\mapsto\ell(x,r)$ for each $x$ has a unique minimizer at $r=0$. 
% Then $g(x,\cdot)$ is uniquely minimized by the CE input $u=\nu(x)$ in~\eqref{eq:ce_input}.
\end{assumption}

Under Assumption~\ref{ass:ce_centered_cost}, the stage cost $g(x,\cdot)$ is uniquely minimized by the CE input $u=\nu(x)$ in~\eqref{eq:ce_input}.
Assumption~\ref{ass:ce_centered_cost} is satisfied by a class of commonly used stage costs. For illustration, we provide the following two examples.
In both cases, $g(x,u)$ is minimized with respect to (w.r.t.) $u$ by the CE input $\nu(x)$.

\begin{example}[Quadratic benchmark]
\label{eg:quadratic}
Suppose that the oracle input in~\eqref{eq:oracle} is affine in $\theta$ and is denoted by 
$\upsilon^*(\theta)=k\theta+b$ with $k\neq0$ and $b\in\Re$. Then the CE input in~\eqref{eq:ce_input} is
$\nu(x)=\upsilon^*(m)=km+b$ for all $x\in X$.

For the quadratic loss $l(\theta,u)=(u-\upsilon^*(\theta))^2$, the stage cost is
\begin{equation}\label{eq:g_quadratic}
	   g(x,u)=E\{(u-\upsilon^*(\theta))^2\}=(u-(km+b))^2+k^2P,
\end{equation}
satisfying~\eqref{eq:cost} with $\ell(x,r)=r^2$ and $\psi(x)=k^2P$. 
%Setting
%$k=1$, $b=0$ yields
%\begin{equation}
%\label{eg:quadratic}
%    g(x,u)=(u-m)^2+P.
%\end{equation}
\end{example}

\begin{example}[Quartic benchmark]
\label{eg:quartic}
Under the same affine oracle input $\upsilon^*(\theta)$, consider the quartic loss
 $l(\theta,u)=(u-\upsilon^*(\theta))^4$. The corresponding  stage cost is
\begin{align}
&g(x,u)=E\{(u-\upsilon^*(\theta))^4\}\label{eq:g_quartic}
\\
=&(u-(km+b))^4+6k^2P(u-(km+b))^2+3k^4P^2, \nonumber
\end{align}
satisfying~\eqref{eq:cost} with $\ell(x,r)=r^4+6k^2Pr^2$ and $\psi(x)=3k^4P^2$.
%Setting $k=1$, $b=0$ yields
%\begin{equation}
%\label{eq:quartic_stage_cost}
%    g(x,u)=(u-m)^4+6P(u-m)^2+3P^2.
%\end{equation}
\end{example}

The belief-state DP problem~\eqref{eq:finite_horizon_bellman} is now fully
specified by the belief state and its dynamics in
Section~\ref{subsec:belief_state_dynamics}, together with the stage cost in
Section~\ref{subsec:risk_to_state}.
Despite this reduction to a two-dimensional state, the backward Bellman
recursion generally does not yield closed-form expressions for the optimal
cost-to-go functions.
This makes it difficult to verify CE optimality directly at every stage.
We therefore examine CE optimality at the level of a single Bellman
recursion, as formulated next.

\subsection{Research Problem}
\label{subsec:research_problem}

The objective of this work is not to solve a particular DP problem. Instead, we study CE optimality at the level of a single Bellman recursion for a generic cost-to-go function \(J:X\to\Re\). 
The function \(J\) may represent the optimal cost-to-go
$J_{t+1}^*$ in~\eqref{eq:finite_horizon_bellman}, or an approximate cost-to-go function used in its place.
We require the following mild well-posedness condition.
\begin{assumption}
	\label{ass:well_defined}
	The function $J:X\to\Re$ satisfies
	\[
	E\big\{\big|J\big(f(x,u,w)\big)\big|\big\}<\infty,
	\qquad
	\forall x\in X,\ u\in U,
	\]
    where $f(x,u,w)$ is given by~\eqref{eq:dynamics_f_def}.
\end{assumption}

Under Assumption~\ref{ass:well_defined}, we define the Bellman objective for a generic cost-to-go function
\begin{equation}
	\label{eq:Q}
	Q_J(x,u) = g(x,u) + E\big\{ J\big(f(x,u,w)\big) \big\}.
\end{equation}
The backward DP recursion in~\eqref{eq:finite_horizon_bellman} is recovered by taking $J=J_{t+1}^*$ at stage $t$.
In what follows, we study~\eqref{eq:Q} under Assumptions~\ref{ass:ce_centered_cost} and~\ref{ass:well_defined}, with $f$ given by~\eqref{eq:dynamics_f_def}.

The input $u$ enters~\eqref{eq:Q} through both terms: it determines the stage cost $g(x,u)$,
which is minimized by the CE input $\nu(x)$ by
Assumption~\ref{ass:ce_centered_cost}, and it determines the expected cost-to-go through the belief-state transition $f$. 
The transition $f$ is given explicitly in~\eqref{eq:dynamics_f_def}, and thus the expected cost-to-go becomes
\begin{equation}\label{eq:expected_future_cost}
	E\big\{ J\big(f(x,u,w)\big) \big\} = E\big\{J\big(m+\sqrt{\eta(P,u)}w,P-\eta(P,u)\big)\big\}.
\end{equation}

We can now formulate CE optimality in terms of the Bellman objective:
when does the expected cost-to-go~\eqref{eq:expected_future_cost} not move the minimizer of the Bellman objective~\eqref{eq:Q} away from the CE input $\nu(x)$? 
Equivalently, when do $g(x,u)$  and $g(x,u)+E\big\{ J\big(f(x,u,w)\big) \big\}$ share the same  minimizer?
This motivates a no-dual-effect (NDE) condition defined from the Bellman-objective side, which we refer to as \emph{Bellman-side NDE}.
This structural question is the focus of the remainder of the paper.\footnote{Establishing policy-level CE optimality for a finite-horizon DP in~\eqref{eq:finite_horizon_bellman} requires this structural CE optimality to hold at every stage, which lies beyond the scope of the present paper.}

\section{Bellman-side No-Dual-Effect}
\label{sec:definition}

This section shifts the NDE notion from the posterior side to the Bellman-objective side.
Section~\ref{subsec:Posterior_NDE} recalls the classical posterior-side NDE notion
and shows that it fails in our setting. Sections~\ref{subsec:dual_effect_function}
and~\ref{subsec:BNDE} introduce the dual-effect function and the resulting
Bellman-side NDE notion. Section~\ref{subsec:CE_sufficient} shows that Bellman-side NDE is a sufficient condition for CE optimality.

\subsection{Posterior-side No-Dual-Effect}
\label{subsec:Posterior_NDE}
We first revisit the posterior-side NDE notion, studied in~\cite{barshalom1974dual,derpich2023dual}.

\begin{definition}[Posterior-side NDE]
\label{def:posterior_side_NDE}
A  partially observed stochastic control problem is said to satisfy posterior-side NDE if the conditional distribution of the estimation error of the unobserved state, given the information history, is independent of the past control inputs.
\end{definition}

Under the corresponding model and cost assumptions, posterior-side NDE is a standard sufficient condition for CE optimality~\cite{barshalom1974dual,derpich2023dual}. 
This NDE notion does not hold, however, in the setting considered here.

\begin{lemma}[Failure of posterior-side NDE]
\label{lem}
The scalar linear-Gaussian model~\eqref{eq:io} does not satisfy posterior-side NDE.
\end{lemma}

\begin{proof}
In the scalar linear-Gaussian model~\eqref{eq:io}, the unobserved state is the parameter 
$\theta$. Then the estimation error satisfies $\theta-m\sim N(0,P)$ conditional on the information. Thus, posterior-side NDE reduces to the requirement that $P$ evolve independently of the input.
However, for every \(P>0\), the map $u\mapsto\frac{P}{1+Pu^2}$ in the variance update~\eqref{eq:P_update_app} is not constant on $U=\Re$. Hence, the input affects the estimation error, and posterior-side NDE does not hold.
\end{proof}

%Although posterior-side NDE fails here, CE remain optimal when the expected cost-to-go~\eqref{eq:expected_future_cost} does not move the minimizer of the Bellman objective~\eqref{eq:Q} away from the CE input.
%This motivates the Bellman-side NDE notion introduced next.

The failure of posterior-side NDE does not by itself imply that the CE input is
not optimal. The posterior may depend on the input while the expected cost-to-go~\eqref{eq:expected_future_cost} does not. In that case, the CE input still
minimizes the Bellman objective~\eqref{eq:Q}. This motivates the
Bellman-side NDE notion introduced next.

\subsection{Dual-Effect Function}
\label{subsec:dual_effect_function}
% Before introducing Bellman-side NDE, 

We first introduce the dual-effect function, the central object on which Bellman-side NDE is defined.

\begin{definition}[Dual-effect function]\label{def:delta_J}
Let Assumption~\ref{ass:well_defined} hold.
For $x\in X$ and $\eta\in[0,P)$, the dual-effect function of a generic cost-to-go $J$ is defined as:
\begin{equation}
        \Delta_J(x,\eta)
    =
    E\!\left\{
        J\bigl(m+\sqrt{\eta}w,P-\eta\bigr)
    \right\}
    -
    J(m,P),\label{eq:delta_J}
\end{equation}
% where this expectation is w.r.t. $w$, defined in~\eqref{eq:w}.
where the expectation is taken w.r.t. $w\sim N(0,1)$, as given in Lemma~\ref{lem:belief_state_transition_eta}.
\end{definition}

This function measures the change in the expected cost-to-go~\eqref{eq:expected_future_cost} when the input induces a variance
reduction $\eta$.
% Recall that the random variable $w$, conditional on $(x,u)$, is standard Gaussian; see Lemma~\ref{lem:belief_state_transition_eta}.
% The function measures the change in the expected cost-to-go when the posterior variance \(P\) reduces to \(P-\eta\).
% Variance reduction implies that the information contained in the belief state increases.
We view the variance reduction as an information gain, so the sign of $\Delta_J(x,\eta)$ carries the following value-of-information (VOI) interpretation:
\begin{enumerate}
\item[(a)] If $\Delta_J(x,\eta)=0$, the variance reduction leaves the expected
cost-to-go unchanged, corresponding to information neutrality.
\item[(b)] If $\Delta_J(x,\eta)<0$, the variance reduction lowers the expected
cost-to-go, so information is beneficial, corresponding to positive VOI.
\item[(c)] If $\Delta_J(x,\eta)>0$, the variance reduction raises the expected
cost-to-go, so information is unfavorable, corresponding to negative VOI.
\end{enumerate}

With $\eta=\eta(P,u)$ as in~\eqref{eq:eta}, 
the dependence of the expected cost-to-go~\eqref{eq:expected_future_cost} on $u$ enters only through $\Delta_J(x,\eta)$, as stated next.

\begin{lemma}[Bellman-objective decomposition]
\label{lem:bellman_decomposition}
Let Assumptions~\ref{ass:ce_centered_cost} and~\ref{ass:well_defined} hold. Then for every $x\in X$ and
$u\in U$,
\begin{equation}\label{eq:exp_cost_in_DEF}
	%    E\!\left\{J\bigl(f(x,u,w)\bigr)\right\}
	E\big\{J\big(m+\sqrt{\eta(P,u)}w,P-\eta(P,u)\big)\big\}
	=J(x)+\Delta_J\bigl(x,\eta(P,u)\bigr).
\end{equation}
Furthermore, the Bellman objective~\eqref{eq:Q} admits the decomposition
\begin{align}
    Q_J(x,u)
%    &=g(x,u)+E\!\left\{J\bigl(f(x,u,w)\bigr)\right\}
%    \nonumber\\
    &=\ell\bigl(x,u-\nu(x)\bigr)+\psi(x)+J(x)
      +\Delta_J\bigl(x,\eta(P,u)\bigr).
    \label{eq:bellman_decomposition}
\end{align}
% and
% \begin{equation}
%         \inf_{u\in U}Q_J(x,u)
%     =\inf_{u\in U}\left[\ell\bigl(x,u-\nu(x)\bigr)
%        +\Delta_J\bigl(x,\eta(P,u)\bigr)\right].
%     \label{eq:new_Q}
% \end{equation}
\end{lemma}

\begin{proof}
%Writing $\eta=\eta(P,u)$, the belief-state
%transition~\eqref{eq:dynamics_f_def} gives
%\[
%    E\!\left\{J\bigl(f(x,u,w)\bigr)\right\}
%    =E\!\left\{J\bigl(m+\sqrt{\eta}\,w,\,P-\eta\bigr)\right\}
%    =J(x)+\Delta_J\bigl(x,\eta\bigr),
%\]
%where the second 
Writing $\eta = \eta(P,u)$,~\eqref{eq:exp_cost_in_DEF} follows from the definition of the dual-effect function in Definition~\ref{def:delta_J}. Substituting this identity together with the
CE-optimal stage cost~\eqref{eq:cost} into~\eqref{eq:Q} yields the
decomposition~\eqref{eq:bellman_decomposition}.
% The infimum then follows since the $u$-independent terms $\psi(x)$ and $J(x)$ do not affect the minimization over $u$.
\end{proof}

\subsection{Bellman-side DE and NDE}
\label{subsec:BNDE}
With the dual-effect function $\Delta_J(x,\eta)$ in~\eqref{eq:delta_J}, we now define Bellman-side NDE, together with its complement, Bellman-side dual-effect (DE).

\begin{definition}[Bellman-side NDE and DE]
\label{def:bellman_de_nde}
Let Assumption~\ref{ass:well_defined} hold, and fix a belief state $x\in X$.
The function $J$ is said to satisfy \emph{Bellman-side NDE at $x$} if
\[
    \Delta_J(x,\eta)=0,\qquad \forall\,\eta\in[0,P).
\]
If instead $\Delta_J(x,\eta)\neq0$ for some $\eta\in[0,P)$, then $J$ has a \emph{Bellman-side DE at $x$}. For a set $D\subseteq X$, $J$ is said to satisfy
Bellman-side NDE on $D$ if it does so at every $x\in D$.
% ; the case $D=X$ gives Bellman-side NDE on $X$.
\end{definition}

A function $J$ may satisfy Bellman-side NDE on part of the belief-state space $X$
and exhibit a Bellman-side DE on the rest, as the following example shows.

\begin{example}[Bellman-side NDE on a subregion]
\label{ex:threshold_statewise_dual_free}
Consider the threshold-type cost-to-go function
\[
    J(m,P)=\max\{P-\gamma,0\},%\qquad \gamma>0.
\]
for some $\gamma>0$.
Since $J$ depends only on the posterior variance, its dual-effect function is
\[
    \Delta_{J}(x,\eta)
    =\max\{P-\eta-\gamma,0\}-\max\{P-\gamma,0\},
    \quad \eta\in[0,P).
\]

If $P\leq\gamma$, then $P-\eta-\gamma\le P-\gamma\leq0$ for all $\eta\in[0,P)$, so
both terms vanish and
\[
    \Delta_{J}(x,\eta)=0,\qquad \forall\,\eta\in[0,P).
\]
By Definition~\ref{def:bellman_de_nde}, $J$ satisfies Bellman-side NDE at every
state with $P\leq\gamma$.

If instead $P>\gamma$, then $\max\{P-\gamma,0\}=P-\gamma$ and
\[
    \Delta_{J}(x,\eta)=\max\{P-\eta-\gamma,0\}-(P-\gamma)<0,
    \qquad \forall\,\eta\in(0,P),
\]
so $J$ exhibits a Bellman-side DE at every state with $P>\gamma$.

In summary, $J$ satisfies Bellman-side NDE on the region
\[
    X_\gamma=\bigl\{x\in X \mid P\leq\gamma\bigr\},
\]
but not on the whole belief-state space $X$.
\end{example}

\subsection{Sufficient Condition for CE Optimality}
\label{subsec:CE_sufficient}

We now return to the CE-optimality problem posed in
Section~\ref{subsec:research_problem}. The notion of Bellman-side NDE allows us to answer the problem directly.

\begin{theorem}[Bellman-side NDE implies CE optimality]
\label{thm:CE_opt_value_side}
Suppose Assumptions~\ref{ass:ce_centered_cost} and~\ref{ass:well_defined} hold and fix $x\in X$. If $J$
satisfies Bellman-side NDE at $x\in X$, then the Bellman objective $Q_J(x,\cdot)$
in~\eqref{eq:Q} is uniquely minimized by the CE input $u=\nu(x)$, defined in~\eqref{eq:ce_input}. Furthermore,
if $J$ satisfies Bellman-side NDE on $X$, then the CE input uniquely minimizes
the Bellman objective at every $x\in X$.
\end{theorem}

\begin{proof}
Fix $x\in X$. Since $J$ satisfies Bellman-side NDE at $x$,
Definition~\ref{def:bellman_de_nde} gives $\Delta_J(x,\eta)=0$ for all
$\eta\in[0,P)$. 
% As $\eta(P,u)\in[0,P)$ for every $u\in U$, it follows that
% $\Delta_J\bigl(x,\eta(P,u)\bigr)=0$ for all $u\in U$. 
Hence the
decomposition~\eqref{eq:bellman_decomposition} in Lemma~\ref{lem:bellman_decomposition} reduces
to
\[
    Q_J(x,u)=\ell\bigl(x,u-\nu(x)\bigr)+\psi(x)+J(x),
    \qquad \forall u\in U.
\]
Since the last two terms are independent of $u$,
\begin{equation}
        \argmin_{u\in U}Q_J(x,u)=\argmin_{u\in U}\ell\bigl(x,u-\nu(x)\bigr).
\end{equation}
By Assumption~\ref{ass:ce_centered_cost}, $\ell(x,\cdot)$ is uniquely minimized
at zero, so $Q_J(x,\cdot)$ is uniquely minimized by the CE input $u=\nu(x)$.
\end{proof}

\begin{remark}
Bellman-side NDE is sufficient but not necessary for CE optimality: the CE input
may minimize $Q_J(x,\cdot)$ even when $\Delta_J(x,\cdot)\neq 0$. This occurs,
for instance, when the expected cost-to-go~\eqref{eq:expected_future_cost} varies with $u$, but is itself
minimized at $u=\nu(x)$, so that CE input is optimal for both terms of $Q_J$.
A simple example is as follows. At a state \(x\in X\), suppose that the CE input satisfies \(\nu(x)=0\).
Let \(     J(m,P)=-P\). Then the expected cost-to-go is
\[
    E\big\{J(f(x,u,w))\big\}
    =
    -\frac{P}{1+Pu^2},
\]
which depends on \(u\), so Bellman-side
NDE fails. However, this term is minimized at \(u=0=\nu(x)\). Since \(g(x,u)\) is also
minimized at \(u=\nu(x)\), the CE input remains a minimizer of the Bellman
objective.
\end{remark}

The dual-effect function $\Delta_J(x,\eta)$ characterizes Bellman-side NDE directly by definition. However, verifying this condition is inconvenient, as it requires evaluating an expectation for every variance reduction $\eta\in[0,P)$. The next section derives an equivalent and more tractable criterion.

\section{Characterization of Bellman-side NDE}
\label{sec:source}

This section develops a characterization of the dual-effect function, Bellman-side DE and NDE, introduced in the previous section.
In Section~\ref{subsec:source_def_regu_ass}, we first state the regularity conditions required for the subsequent analysis. 
In Section~\ref{subsec:from_source_to_de}, we then introduce
the \emph{source} associated with a generic cost-to-go function
and derive a representation of the dual-effect function
in terms of this source.
Using this representation, we characterize VOI with the source in Section~\ref{subsec:voi_source}. 
In Section~\ref{subsec:source_NDE}, we show that Bellman-side NDE is equivalent to the source vanishing identically, and use this to characterize CE-optimality.

\subsection{Regularity Conditions}
\label{subsec:source_def_regu_ass}

We first impose the following regularity conditions.
\begin{assumption}[Regularity conditions]
\label{ass:regularity}
The function $J: X \to \Re$ 
satisfies:
\begin{enumerate}
    \item[(a)] 
$J$ has continuous partial derivatives $\partial_P J$, $\partial_m J$, and $\partial_{mm} J$ on $X$, where
$\partial_P J$ denotes the first-order partial derivative w.r.t. $P$, and
$\partial_m J$ and $\partial_{mm} J$ denote the first- and second-order partial derivatives with
respect to $m$, respectively.
    
    \item[(b)] (Locally uniform polynomial growth.) 
    For every compact interval $\mathcal I \subset (0, \infty)$, 
    there exist constants $c_{\mathcal I} > 0$ and 
    $q_{\mathcal I} \geq 0$ such that, for all $m \in \Re$ 
    and all $P \in \mathcal I$,
    \[
        % \max\bigl\{|J(m, P)|, |\partial_m J(m, P)|, 
        % |\partial_{mm} J(m, P)|, |\partial_P J(m, P)|\bigr\}
        % \leq c_{\mathcal I}\bigl(1 + |m|^{q_{\mathcal I}}\bigr),
        \max\bigl\{|J|, |\partial_m J|, 
        |\partial_{mm} J|, |\partial_P J|\bigr\}(m, P)
        \!\leq\! c_{\mathcal I}\bigl(1 + |m|^{q_{\mathcal I}}\bigr).    \]
     % where \(\partial_m J\) denotes the first-order partial derivative of \(J\) w.r.t. \(m\).
\end{enumerate}
\end{assumption}

% \begin{remark}[What the source measures]
% \label{rem:what_source_measures}
% The source \(\mathcal S_J\) compares two effects generated by Bayesian update. When the control $u$ induces a variance reduction \(\eta\), the next belief state is $(m+\sqrt{\eta}\,w, P-\eta)$.
% It implies that information gain does not merely reduce posterior variance; it injects additional dispersion into the posterior mean. 
% Thus, a balance exists between $\partial_P J$, which measures the value of reducing $P$, and $\partial_{mm} J$, which captures the effect of the induced random dispersion of $m$.  
% % A zero source means that the two effects exactly offset each other.
% \end{remark}

\begin{remark}
%	Assumption~\ref{ass:regularity} is required only for the source characterization; 
	Recall that Bellman-side NDE in Definition~\ref{def:bellman_de_nde} needs only Assumption~\ref{ass:well_defined}.
	The threshold cost-to-go of Example~\ref{ex:threshold_statewise_dual_free}, for instance, is Bellman-side NDE on $X_\gamma$ but fails Assumption~\ref{ass:regularity}(a) at $P=\gamma$.
	Thus, the results of this section do not apply to Example~\ref{ex:threshold_statewise_dual_free}.
\end{remark}

We next state several consequences of Assumption~\ref{ass:regularity} that will
be used in the subsequent analysis.

\begin{lemma}[Consequences of regularity conditions]
\label{lem:regularity_consequences}
Suppose Assumption~\ref{ass:regularity} holds. Fix $x=(m,P)\in X$ and
$\bar\eta\in(0,P)$. Then there exists a continuous function
$G:\Re\to[0,\infty)$ of polynomial growth such that
\[
    E\{|w|^kG(w)\}<\infty,
    \qquad \forall\,k\geq0,
    \qquad w\sim N(0,1),
\]
and for all $\eta\in[0,\bar\eta]$ and $w\in\Re$, 
\begin{equation}
    \max\bigl\{|J|,|\partial_m J|,|\partial_{mm} J|,|\partial_P J|\bigr\}
    (m+\sqrt{\eta}\,w,P-\eta)
    \leq G(w).
    \label{eq:domination}
\end{equation}
Define
\[
    F_x(\eta)
    =
    E\{J(m+\sqrt{\eta}\,w,P-\eta)\},
    \qquad
    \eta\in[0,\bar\eta].
\]
This uniform domination further yields the following consequences:
\begin{enumerate}
    \item[(a)] The expected cost-to-go is well defined for every
    $\eta\in[0,\bar\eta]$, since
    \[
        E\{|J(m+\sqrt{\eta}\,w,P-\eta)|\}<\infty,
        \qquad
        \forall\,\eta\in[0,\bar\eta].
    \]
    Since $x$ and $\bar\eta$ are arbitrary, this also implies
    Assumption~\ref{ass:well_defined}.

    \item[(b)] The mapping
    \(
        \eta\mapsto F_x(\eta)
    \)
    is differentiable on $(0,\bar\eta)$, with
    % \begin{equation}
    %     % \frac{\partial}{\partial\eta}F_x(\eta)
    %     F_x'(\eta)
    %     =
    %     E\left\{
    %         \frac{w}{2\sqrt{\eta}}
    %         \partial_m J(m+\sqrt{\eta}\,w,P-\eta)
    %         -
    %         \partial_P J(m+\sqrt{\eta}\,w,P-\eta)
    %     \right\},
    %     \label{eq:diff_under_int}
    % \end{equation}
    \begin{equation}
    \begin{aligned}
        F_x'(\eta)
        &=
        E\Biggl\{
            \frac{w}{2\sqrt{\eta}}
            \partial_m J(m+\sqrt{\eta}\,w,P-\eta) \\
        &\qquad {}-
            \partial_P J(m+\sqrt{\eta}\,w,P-\eta)
        \Biggr\},
    \end{aligned}
    \label{eq:diff_under_int}
\end{equation}
    where \(F_x'\) denotes the derivative of \(F_x\) w.r.t. \(\eta\). 

    \item[(c)] For every $\eta\in(0,\bar\eta]$, Stein's identity
    \cite[Sec.~14.2]{chen2011normal} gives
    % \begin{equation}
    %     E\{w\partial_m J(m+\sqrt{\eta}\,w,P-\eta)\}
    %     =
    %     \sqrt{\eta}\,
    %     E\{\partial_{mm} J(m+\sqrt{\eta}\,w,P-\eta)\}.
    %     \label{eq:stein_id}
    % \end{equation}
    \begin{equation}
\label{eq:stein_id}
\begin{aligned}
    &E\{w\partial_m J(m+\sqrt{\eta}\,w,P-\eta)\} \\
    &\quad =
    \sqrt{\eta}\,
    E\{\partial_{mm} J(m+\sqrt{\eta}\,w,P-\eta)\}.
\end{aligned}
\end{equation}

    \item[(d)] The function $F_x$ is continuous on $[0,\bar\eta]$, with
    \(
        F_x(0)=J(m,P).
    \)

    \item[(e)] For every $\eta\in(0,\bar\eta)$,
    \[
        |F_x'(\eta)|
        \leq
        \frac{3}{2}E\{G(w)\}.
    \]
    Hence $F_x$ is Lipschitz continuous, and then absolutely continuous,
    on $[0,\bar\eta]$.
\end{enumerate}
\end{lemma}

\begin{proof}
See Appendix~\ref{app:regularity_proof}.
\end{proof}

\subsection{Source Representation of Dual-Effect Function}
\label{subsec:from_source_to_de}
With the regularity conditions in Assumption~\ref{ass:regularity}, we define the \emph{source}\footnote{The terminology is from the heat equation: $\mathcal S_J$
is the source term in the inhomogeneous equation
$\partial_P J-\frac{1}{2}\partial_{mm} J=\mathcal S_J$.} $\mathcal S_J:X\to\Re$ by
\begin{equation}\label{eq:source_eq}
        \mathcal S_J(m,P)
    =
    \partial_P J(m,P)-\frac12 \partial_{mm} J(m,P).
\end{equation}
We state that the source measures the local sensitivity of the expected cost-to-go~\eqref{eq:expected_future_cost} to variance reduction $\eta$ in the following theorem.

\begin{theorem}[Source representation of dual-effect function]
\label{thm:exact_source_representation}
Suppose Assumption~\ref{ass:regularity} holds. 
Then, for every \(x\in X\)  and every \(\eta\in(0,P)\),
\[ 
    \frac{\partial}{\partial \eta}
    E\!\left\{
        J(m+\sqrt{\eta}w,P-\eta)
    \right\}
    =
    -
    E\!\left\{ 
        \mathcal S_J(m+\sqrt{\eta}w,P-\eta)
    \right\}.
\]
Furthermore, for every \(\eta\in[0,P)\),
\begin{equation}
\begin{aligned}
        \Delta_J(x,\eta)&=\label{eq:source_rep}
        E\!\left\{
        J(m+\sqrt{\eta}w,P-\eta)
    \right\}
    -
    J(m,P)\\
    &=
    -
    \int_0^\eta
    E\!\left\{
        \mathcal S_J(m+\sqrt{s}w,P-s)
    \right\}ds.
\end{aligned}
\end{equation}
\end{theorem}

\begin{proof}
Fix $\eta\in(0,P)$ and choose $\bar\eta\in(\eta,P)$. By
Lemma~\ref{lem:regularity_consequences}(b)--(c), for every
$s\in(0,\bar\eta)$,
% \begin{align}
%     F_x'(s)
%     &=
%     E\left\{
%         \frac{w}{2\sqrt{s}}
%         \partial_m J(m+\sqrt{s}\,w,P-s)
%         -
%         \partial_P J(m+\sqrt{s}\,w,P-s)
%     \right\}
%     \label{eq:source_derivative_1}\\
%     &=
%     E\left\{
%         \frac{1}{2}\partial_{mm} J(m+\sqrt{s}\,w,P-s)
%         -
%         \partial_P J(m+\sqrt{s}\,w,P-s)
%     \right\}
%     \label{eq:source_derivative_2}\\
%     &=
%     -E\left\{
%         \mathcal S_J(m+\sqrt{s}\,w,P-s)
%     \right\},
%     \label{eq:source_derivative_3}
% \end{align}
\begin{align}
    F_x'(s)
    &= E\Bigl\{
        \frac{w}{2\sqrt{s}}
        \partial_m J(m+\sqrt{s}\,w,P-s)
    \notag\\
    &\qquad {}-
        \partial_P J(m+\sqrt{s}\,w,P-s)
    \Bigr\}
    \label{eq:source_derivative_1}\\
    &= E\Bigl\{
        \frac{1}{2}
        \partial_{mm} J(m+\sqrt{s}\,w,P-s)
    \notag\\
    &\qquad {}-
        \partial_P J(m+\sqrt{s}\,w,P-s)
    \Bigr\}
    \label{eq:source_derivative_2}\\
    &= -E\Bigl\{
        \mathcal S_J(m+\sqrt{s}\,w,P-s)
    \Bigr\},
    \label{eq:source_derivative_3}
\end{align}
where the first equality follows from~\eqref{eq:diff_under_int}, the second
from Stein's identity~\eqref{eq:stein_id}, and the third from the definition
of the source $\mathcal S_J$.

By Lemma~\ref{lem:regularity_consequences}(d)--(e), $F_x$ is absolutely
continuous on $[0,\bar\eta]$ and satisfies
\(
    F_x(0)=J(m,P)
\).
Hence, the fundamental theorem of calculus gives
\[
\begin{aligned}
    \Delta_J(x,\eta)
    &=
    F_x(\eta)-F_x(0)=
    \int_0^\eta F_x'(s)\,ds\\
    &=
    -\int_0^\eta
    E\left\{
        \mathcal S_J(m+\sqrt{s}\,w,P-s)
    \right\}\,ds,
\end{aligned}
\]
which proves~\eqref{eq:source_rep}.
\end{proof}

This representation replaces the direct evaluation of
$\Delta_J(x,\eta)$ for every $\eta\in[0,P)$ by the analysis of the local
differential quantity $\mathcal S_J$. 
The subsequent sections exploit this
result in two ways. Section~\ref{subsec:voi_source} relates the local and global signs of $\Delta_J(x,\eta)$ to corresponding sign conditions on
$\mathcal{S}_J$.
Section~\ref{subsec:source_NDE} characterizes Bellman-side NDE by the zero-source condition $\mathcal S_J\equiv0$.

\subsection{Source Characterization of Value of Information}
\label{subsec:voi_source}

Theorem~\ref{thm:exact_source_representation} identifies the source $\mathcal S_J$
as the quantity that governs $\Delta_J$, i.e., how a variance reduction changes the expected cost-to-go~\eqref{eq:expected_future_cost}. Considering a variance reduction as an information gain, the source thus determines the VOI. We first describe the VOI at a single belief state, then extend it to the whole state space.

\begin{lemma}[Source and marginal VOI]
\label{lem:local_voi}
Suppose Assumption~\ref{ass:regularity} holds and fix $x\in X$. Then
\begin{equation}\label{eq:local_expansion}
    \Delta_J(x,\eta)=-\eta\,\mathcal S_J(m,P)+o(\eta),
    \qquad \eta\downarrow0,
\end{equation}
and the sign of $\mathcal S_J(m,P)$ determines the first-order effect of a small
variance reduction on the expected cost-to-go~\eqref{eq:expected_future_cost}:
\begin{enumerate}
    \item[(a)] (Neutral to first order.) If $\mathcal S_J(m,P)=0$, then
    $\Delta_J(x,\eta)=o(\eta)$ as $\eta\downarrow0$.
    \item[(b)] (Locally positive VOI.) If $\mathcal S_J(m,P)>0$, then there
    exists $\eta_0>0$ such that $\Delta_J(x,\eta)<0$ for all $\eta\in(0,\eta_0)$.
    \item[(c)] (Locally negative VOI.) If $\mathcal S_J(m,P)<0$, then there
    exists $\eta_0>0$ such that $\Delta_J(x,\eta)>0$ for all $\eta\in(0,\eta_0)$.
\end{enumerate}
\end{lemma}

\begin{proof}
    See Appendix~\ref{app:local_voi}.
\end{proof}

By the expansion~\eqref{eq:local_expansion}, $\mathcal S_J(m,P)$ is the \emph{marginal VOI} at $x$: its sign determines whether the VOI is positive,
neutral to first order, or negative; its magnitude gives how much the expected cost-to-go~\eqref{eq:expected_future_cost} changes per unit of variance reduction.
When the source has a constant sign across the whole state space $X$, this marginal VOI extends to a global one, holding for every variance reduction, as stated next.% rather than only to first order.

\begin{corollary}[Source sign and global VOI on $X$]
\label{cor:zero_source}
Suppose Assumption~\ref{ass:regularity} holds.
\begin{enumerate}
    \item[(a)] (Information neutrality.) If $\mathcal S_J(m,P)=0$ for all
    $x\in X$, then
    \[
        \Delta_J(x,\eta)=0,\qquad \forall\,\eta\in(0,P).
    \]
    \item[(b)] (Positive VOI.) If $\mathcal S_J(m,P)>0$ for all $x\in X$,
    then
    \[
        \Delta_J(x,\eta)<0,\qquad \forall\,\eta\in(0,P).
    \]
    \item[(c)] (Negative VOI.) If $\mathcal S_J(m,P)<0$ for all $x\in X$,
    then
    \[
        \Delta_J(x,\eta)>0,\qquad \forall\,\eta\in(0,P).
    \]
\end{enumerate}
\end{corollary}

\begin{proof}
Fix \(x\in X\) and \(\eta\in(0,P)\). By
Theorem~\ref{thm:exact_source_representation},
\begin{align}
    \Delta_J(x,\eta)&=E\big\{
            J\big(m+\sqrt{\eta}w,\ P-\eta\big)
        \big\}
        -
        J(m,P)\nonumber\\
    &=
    -
    \int_0^\eta
    E\big\{
        \mathcal S_J\big(m+\sqrt{s}w,\ P-s\big)
    \big\}\,ds.\nonumber
\end{align}
If \(\mathcal S_J\equiv0\) on \(X\), then the integrand is zero for all
\(s\in[0,\eta]\), and hence $\Delta_J(x,\eta)=0$. This proves (a).

If \(\mathcal S_J>0\) on \(X\), then the integrand
\[
    E\big\{
        \mathcal S_J\big(m+\sqrt{s}w,\ P-s\big)
    \big\}
    >0,
    \qquad
    0\le s\le\eta.
\]
Thus, $\Delta_J(x,\eta)<0$, and this proves (b).

If \(\mathcal S_J< 0\) on \(X\), then the integrand
\[
    E\big\{
        \mathcal S_J\big(m+\sqrt{s}w,\ P-s\big)
    \big\}
    < 0,
    \qquad
    0\le s\le\eta.
\]
Thus, $\Delta_J(x,\eta)>0$,  and this proves (c).
\end{proof}

Here we illustrate the above global VOI with the following example.

\begin{example}[Quadratic cost-to-go function]
\label{eg:quadratic_cont}
Consider the quadratic cost-to-go function
$   J(m,P)
    =
    m^2+\lambda P,
$
where $\lambda\in\Re$. Its source is
$$
\begin{aligned}
    \mathcal S_J(m,P)
    =
    \partial_P J-\frac12 \partial_{mm} J 
    =
    \lambda-1.
\end{aligned}
$$
Thus, $\lambda>1$ corresponds to a positive-VOI function
$J$ on $X$, $\lambda=1$ corresponds to an information-neutral function $J$ on $X$, and
$\lambda<1$ corresponds to a negative-VOI function $J$ on $X$.
This follows immediately from Corollary~\ref{cor:zero_source}, and we will verify the result in numerical studies in Section~\ref{sec:numerical}.
% , without computing the following Bellman objective
% \begin{equation}
%     \begin{aligned}
% \label{eq:J0}
%     % J^*(m,P)
%     % &=
%     % &\inf_{u\in\Re}
%     % \left[
%     %     g(x,u)
%     %     +
%     %     E\big\{
%     %         J\big(m+\sqrt{\eta(P,u)}w, P-\eta(P,u)\big)
%     %     \big\}
%     % \right]\\
%     %     =&
%     \inf_{u\in\Re}
%     \left[
%         g(x,u)
%         +
%         E\big\{
%             (m+\sqrt{\eta(P,u)}w)^2
%             +
%             \lambda\big(P-\eta(P,u)\big)
%         \big\}
%     \right].
% \end{aligned}
% \end{equation}
\end{example}

% Beyond the sign information, Corollary~\ref{lem:local_voi} also shows that the
% magnitude of the source determines the local marginal value of variance reduction, or equivalently, the local marginal value of information.

\subsection{Source Characterization of Bellman-side NDE and CE Optimality}
\label{subsec:source_NDE}

We now show that, for the scalar linear-Gaussian model~\eqref{eq:io},
Bellman-side NDE on $X$ is equivalent to a vanishing source on $X$, and that this
condition is sufficient for CE optimality.

\begin{theorem}[Source characterization of Bellman-side NDE]
\label{thm:zero_dual_free_eqiv}
Suppose Assumption~\ref{ass:regularity} holds. Then \(J\) is
Bellman-side NDE on $X$ if and only if
\[
    \mathcal S_J(m,P)=0,
    \qquad \forall x\in X.
\]
If this condition holds and Assumption~\ref{ass:ce_centered_cost} is satisfied, the CE input in~\eqref{eq:ce_input} uniquely minimizes the Bellman objective~\eqref{eq:Q} at every $x\in X$.
\end{theorem}

\begin{proof}
The sufficiency direction follows directly from
Corollary~\ref{cor:zero_source}: if \(\mathcal S_J\equiv0\) on \(X\), then
\[
    E\{J(m+\sqrt{\eta}w,P-\eta)\}=J(m,P),
    \quad \forall x\in X,\  \forall \eta\in[0,P),
\]
so \(J\) is Bellman-side NDE on $X$.

Conversely, suppose that \(J\) is Bellman-side NDE on $X$.
For a fixed \(x\in X\) and for all $ \eta\in[0,P)$
\begin{equation}\label{eq:nde_def}
    \Delta_J(x,\eta)= E\{J(m+\sqrt{\eta}w,P-\eta)\}-J(m,P)=0.
\end{equation}
On the other hand, Lemma~\ref{lem:local_voi} gives the local expansion
\begin{align}
        \Delta_J(x,\eta)&=E\{J(m+\sqrt{\eta}w,P-\eta)\}-J(m,P) \nonumber\\
    &=
    -\eta\,\mathcal S_J(m,P)+o(\eta),
    \qquad \eta\downarrow0.\nonumber
\end{align}
Combining this with~\eqref{eq:nde_def} gives
\( 
    0=-\eta\,\mathcal S_J(m,P)+o(\eta).
\)
Dividing it by \(\eta\) and letting \(\eta\downarrow0\) yields
\( 
    \mathcal S_J(m,P)=0.
\)
Since \(x\in X\) was arbitrary, \(\mathcal S_J\equiv0\) on \(X\).
The final claim of CE optimality follows directly from Theorem~\ref{thm:CE_opt_value_side}.
\end{proof}
% We note that \(J(m,P)=m^2+P\) is the posterior second moment of the unknown parameter, i.e.,
% \(J(m,P)=E[\theta^2\mid x]\).
% Its Bellman-side NDE can therefore also be understood as a consequence of the tower property of conditional expectation, as stated in the following lemma.

\section{Bellman-side NDE Function Space}
\label{sec:function_space}

The previous section characterized Bellman-side NDE on $X$ by the zero-source condition $\mathcal S_J\equiv0$. 
A natural question is therefore which functions $J$ satisfy this condition.
% , i.e., the structure of the entire class of Bellman-side NDE functions. 
This section answers that question.
Section~\ref{subsec:representation_theorem} develops the Gaussian convolution representation of Bellman-side NDE  functions and the resulting function space, while Section~\ref{subsec:poly_NDE_fun_space} characterizes
the polynomial subspace.

\subsection{Gaussian convolution representation of Bellman-side NDE  functions}
\label{subsec:representation_theorem}

We introduce the \emph{complete-information boundary function}
\(\phi_J:\Re\to\Re\), defined as the $P\downarrow0$ boundary trace\footnote{The term \emph{boundary trace} is standard in the partial differential equation literature,
	whereas \emph{complete-information boundary function} emphasizes its control
	interpretation.} of \(J\),
\begin{equation}
    \phi_J(m)=\lim_{P\downarrow0}J(m,P),
    \qquad m\in\Re,
    \label{eq:phi-J}
\end{equation}
% the limit being locally uniform in \(m\). 
when the limit on the right exists.
This boundary function carries a
direct control interpretation. As \(P\downarrow0\) the posterior uncertainty
vanishes and the posterior collapses to a point mass at \(m\). Thus, the unknown
parameter is then known exactly and equals the posterior mean. Accordingly,
\(\phi_J(m)\) is the cost-to-go under complete information about the parameter.
As a notational shorthand we write \(J(m,0)=\phi_J(m)\). Since \(X=\Re\times\Re_{++}\) is open, this is a convention for the limit as $P\downarrow0$  and does not assume that \((m,0)\in X\).

\begin{assumption}%[Complete-information boundary function]
\label{ass:boundary_trace}
The function \(J\) satisfies
% Assumption~\ref{ass:regularity} and, in addition,
% the following conditions.
\begin{enumerate}
    % \item[(a)] 
    % % \emph{Complete-information boundary function.}
    % The boundary trace~\eqref{eq:phi-J} exists: there is a
    % function \(\phi_J:\Re\to\Re\) such that
    % \[
    %     \lim_{P\downarrow0}J(m,P)=\phi_J(m)
    %     \qquad\text{locally uniformly in } m\in\Re.
    % \]
    \item[(a)] The convergence in~\eqref{eq:phi-J} is locally uniform in \(m\), i.e., for every compact
    interval \(K\subset\Re\),
    \[
        \sup_{m\in K}|J(m,P)-\phi_J(m)|\to0,
        \qquad P\downarrow0.
    \]
    \item[(b)] \(J\) satisfies a polynomial growth bound uniformly near
    \(P=0\): for every \(\bar P>0\), there exist constants \(c>0\)
    and \(q\ge0\) such that
    \[
        |J(m,P)|\le c(1+|m|^q),
        \qquad \forall m\in\Re,\ P\in(0,\bar P].
    \]
\end{enumerate}
\end{assumption}

% The complete-information boundary function \(\phi_J(m)=J(m,0)\) is the function under
% complete information. 
% Theorem~\ref{thm:widder_representation}
% shows that this complete-information boundary function has a unique Bellman-side NDE Gaussian convolution:
% its posterior expectation under the current belief. Thus, before stating
% the result, the central message is 
% NDE value is equal to posterior expectation of complete-information boundary function, as stated in the following proposition.

Throughout the remainder of the paper, we work with the function class
\begin{equation}
\label{eq:H}
    \mathcal H
    =
    \{J:X\to\Re \mid J \text{ satisfies Assumptions~\ref{ass:regularity} and~\ref{ass:boundary_trace}}\}.
\end{equation}
Within $\mathcal H$, we first define the Bellman-side NDE function space by~Theorem~\ref{thm:zero_dual_free_eqiv},
\begin{equation}
\label{eq:H-NDE}
    \mathcal H_{\mathrm{NDE}}
    =
    \{J\in\mathcal H \mid \mathcal S_J(x)=0,\ \forall x\in X\}.
\end{equation}
We next show that, for any $J\in\mathcal H$, its complete-information boundary function
inherits two properties: continuity and polynomial growth.
We collect these functions in the class
\begin{equation}
    \begin{aligned}
        \mathcal B
        =&
        \bigl\{\phi:\Re\to\Re \mid
        \phi \text{ is continuous}, \\
        &\qquad
        |\phi(m)|\le c_\phi(1+|m|^{q_\phi})
        \text{ for some } c_\phi>0,\ q_\phi\ge0\bigr\},
    \end{aligned}
    \label{eq:B}
\end{equation}
% Under the above assumptions, the complete-information boundary function inherits two properties: continuity and polynomial growth, that are used in the representation results of this section.
% \begin{lemma}[Properties of the complete-information boundary function]
% 	\label{lem:phi_\partial_P Jroperties}
% 	Suppose $J$ satisfies Assumptions~\ref{ass:regularity} and
% 	\ref{ass:boundary_trace}. Then $\phi_J$, defined
% 	in~\eqref{eq:phi-J}, is continuous on $\Re$ and of polynomial
% 	growth: there exist constants $c>0$ and $q\ge0$ such that
% 	\[
% 	|\phi_J(m)|\le c\,(1+|m|^{q}),
% 	\qquad m\in\Re .
% 	\]
% \end{lemma}
and define two maps acting in opposite directions between $\mathcal H$ and $\mathcal B$: the boundary trace $J\mapsto\phi_J$ and the Gaussian convolution
$\phi\mapsto J_\phi$, which comes shortly.

\begin{lemma}[Properties of $\phi_J$]
\label{lem:H_to_B}
Let $J\in\mathcal H$. Then $\phi_J$, defined in \eqref{eq:phi-J}, is continuous
on $\Re$ and of polynomial growth, i.e., $\phi_J\in\mathcal B$.
\end{lemma}
\begin{proof}
See Appendix~\ref{app:H_to_B}.
\end{proof}

The preceding lemma shows that the boundary trace maps $\mathcal H$ into
$\mathcal B$. Conversely, starting from any $\phi\in\mathcal B$, the next
lemma constructs a function $J_\phi\in\mathcal H$ by Gaussian convolution.

\begin{lemma}[Gaussian convolution]
\label{lem:B_to_H}
Let $\phi\in\mathcal B$ and define
\[
    J_\phi(m,P)
    =
    E\{\phi(m+\sqrt P\,z)\},
    \qquad z\sim N(0,1).
\]
Then $J_\phi\in\mathcal H$, $\mathcal S_{J_\phi}\equiv0$ on $X$, and
$\phi_{J_\phi}=\phi$. In particular, $J_\phi\in\mathcal H_{\mathrm{NDE}}$.
\end{lemma}

\begin{proof}
See Appendix~\ref{app:B_to_H}.
\end{proof}

With Lemmas~\ref{lem:H_to_B} and~\ref{lem:B_to_H}, the following theorem shows that the Bellman-side NDE functions admit a Gaussian convolution representation.
% restricted to
% $\mathcal H_{\mathrm{NDE}}$, these two maps are mutually inverse.

\begin{theorem}[Gaussian convolution representation]
\label{thm:widder_representation}
Let \(J\in\mathcal H\), and let \(\phi_J\) denote its
complete-information boundary function. Then \(J\) is Bellman-side NDE on
\(X\) if and only if
\[
    J
    =
    J_{\phi_J}
    \qquad\text{on }X,
\]
i.e.,
\[
    J(m,P)
    =
    E\big\{
        \phi_J(m+\sqrt P\,z)
    \big\}
    =
    E\big\{
        \phi_J(\theta)\mid x
    \big\},
\]
where \(z\sim N(0,1)\) and
\(\theta\mid x\sim N(m,P)\).
\end{theorem}

\begin{proof}
See Appendix~\ref{app:widder_representation}.
\end{proof}

 % Statement~(c) is constructive: it represents every Bellman-side NDE function as the posterior expectation of its complete-information boundary function. 
 This representation motivates the following construction of the Bellman-side NDE function space.

\begin{lemma}[Bellman-side NDE function space]
\label{lem:NDE_space}
The Bellman-side NDE function space admits the representation
\[
    \mathcal H_{\mathrm{NDE}}
    =
    \left\{
        (m,P)\mapsto E\bigl\{\phi(m+\sqrt P\,z)\bigr\}
        \;\middle|\;
        \phi\in\mathcal B
    \right\},
\]
where $z\sim N(0,1)$. Moreover, each \(J\in\mathcal H_{\mathrm{NDE}}\) is represented by a unique
\(\phi\in\mathcal B\), i.e., its complete-information boundary function
\(\phi_J\).
\end{lemma}

\begin{proof}
We prove the two inclusions. Let \(J\in\mathcal H_{\mathrm{NDE}}\). Then
\(J\in\mathcal H\) and \(J\) is Bellman-side NDE on \(X\). By
Lemma~\ref{lem:H_to_B}, its complete-information boundary function satisfies
\(\phi_J\in\mathcal B\), and Theorem~\ref{thm:widder_representation} gives
\( 
    J
    =
    J_{\phi_J},
\)
i.e.,
\[
    J(m,P)
    =
    E\bigl\{
        \phi_J(m+\sqrt P\,z)
    \bigr\},
    \qquad z\sim N(0,1).
\]
Hence \(J\) belongs to the set on the right-hand side.

Conversely, let \(\phi\in\mathcal B\) and define
\[
    J_\phi(m,P)
    =
    E\bigl\{
        \phi(m+\sqrt P\,z)
    \bigr\},
    \qquad z\sim N(0,1).
\]
By Lemma~\ref{lem:B_to_H},
\(J_\phi\in\mathcal H_{\mathrm{NDE}}\). This proves the reverse inclusion and
hence the stated representation.

It remains to show uniqueness. Suppose that
\(
    J=J_{\phi}=J_{\psi}
\)
for some \(\phi,\psi\in\mathcal B\). By
Lemma~\ref{lem:B_to_H},
\(
    \phi_{J_\phi}=\phi
\), \(
    \phi_{J_\psi}=\psi.
\)
Since both \(J_\phi\) and \(J_\psi\) equal \(J\), their
complete-information boundary function is \(\phi_J\). Therefore
\(
    \phi=\phi_J=\psi.
\)
Thus the representing boundary function is unique and is precisely
\(\phi_J\).
\end{proof}

\subsection{Polynomial Bellman-side NDE Function Space}
\label{subsec:poly_NDE_fun_space}

To make the abstract Bellman-side NDE function space $\mathcal H_{\mathrm{NDE}}$ more concrete, we now construct a polynomial subspace.
For \(n\ge0\), define
\begin{equation}\label{eq:poly_basis}
        h_n(m,P)
    =
    E\{
        (m+\sqrt P\,z)^n
    \},
    \qquad
    z\sim N(0,1).
\end{equation}
The functions $h_n$  are the classical \emph{heat polynomials} of Rosenbloom and Widder~\cite{rosenbloom1959expansions}\footnote{%
	Reference~\cite{rosenbloom1959expansions} normalizes the heat equation as
	$u_t=u_{xx}$, whereas the belief-state dynamics give $\partial_P J=\tfrac12 \partial_{mm} J$.
	The two are related by	$t=P/2$, so that
	$h_n(m,P)=v_n(m,P/2)$, where $v_n$ denotes the $n$-th heat polynomial
	of~\cite{rosenbloom1959expansions}.} in the belief coordinates $(m,P)$.
The first few basis functions are
$
    h_0=1$, $
    h_1=m$, $
    h_2=m^2+P$, $
    h_3=m^3+3mP$, $
    h_4=m^4+6m^2P+3P^2$.
Each \(h_n\) satisfies zero source over $X$, i.e., $\mathcal S_{h_n}\equiv0$. Indeed, \(h_n\) is the Gaussian convolution of the polynomial boundary function
\(\phi(m)=m^n\), and hence belongs to
\(\mathcal H_{\mathrm{NDE}}\).

We define the finite-dimensional subspace generated by the polynomial basis functions in~\eqref{eq:poly_basis} up to degree \(n\) as
\[
    \mathcal H_n
    =
    \operatorname{span}\{h_0,\ldots,h_n\}.
\]
The corresponding polynomial Bellman-side NDE space is given by
\[
    \mathcal H_{\mathrm{POLY}}
    =
    \cup_{n\ge0}\mathcal H_n
    =
    \operatorname{span}\{h_n\mid n\ge0\},
\]
which consists of all finite-degree polynomial Bellman-side NDE functions. %, and does not include non-polynomial ones.
% , even when they admit convergent power-series expansions with infinitely many nonzero terms.
% Thus, \(\mathcal H_n\) contains all Bellman-side NDE functions generated by complete-information boundary polynomials of degree at most \(n\), while \(\mathcal H_{\mathrm{poly}}\) collects all polynomial Bellman-side NDE functions.
These spaces form the nested hierarchy
\[
    \mathcal H_0
    \subset  %\subsetneq
    \mathcal H_1
    \subset  %\subsetneq
    \cdots
    \subset  %\subsetneq
    \mathcal H_{\mathrm{POLY}}
    \subset  %\subsetneq
    \mathcal H_{\mathrm{NDE}},
\]
where $\subset$ denotes strict inclusion. In particular, the last inclusion implies that Bellman-side NDE functions on $X$ need not arise from polynomial boundary functions. The following example gives a non-polynomial Bellman-side NDE function on $X$.

\begin{example}[A non-polynomial Bellman-side NDE function]
Let
\[
    \phi_J(\theta)
    =
    \cos(a\theta),
    \qquad
    a\neq 0.
\]
Then \(\phi_J\) is continuous and bounded, and hence belongs to \(\mathcal B\).
Its Gaussian convolution is
\[
    J(m,P)
    =
    E\big\{
        \cos\big(a(m+\sqrt P\,z)\big)
    \big\}
    =
    e^{-\frac12a^2P}\cos(am).
\]
A direct calculation gives
\[
    \partial_P J
    \!=\!
    -\frac12a^2e^{-\frac12a^2P}\cos(am),
\ 
    \partial_{mm} J
    \!=\!
    -a^2e^{-\frac12a^2P}\cos(am).
\]
Thus \(J\in\mathcal H_{\mathrm{NDE}}\) due to 
\( 
    \mathcal S_{J}
    =
    \partial_P J-\frac12\partial_{mm} J
    =
    0
\) for each $x\in X$.
% But its complete-information boundary function \(\cos(am)\) is not a polynomial, \(J\notin\mathcal H_{\mathrm{POLY}}\).
% Although its complete-information boundary function $\cos(am)$ admits an infinite power-series expansion, it is not a polynomial of finite degree and hence does not belong to $\mathcal H_{\mathrm{POLY}}$.
However, $J(m,P)$
is not a polynomial in \((m,P)\) of finite degree, and hence
\(J\notin\mathcal H_{\mathrm{POLY}}\).
\end{example}

A Bellman-side NDE function has a complete-information boundary that is a
polynomial of degree at most $n$ if and only if this function lies in the
$(n+1)$-dimensional space $\mathcal H_n=\operatorname{span}\{h_0,\dots,h_n\}$, as
stated next.

\begin{lemma}[Heat-polynomial characterization] % of polynomial NDE functions]
\label{lem:polynomial_NDE_basis}
% Suppose $J$ satisfies the assumptions of
% Theorem~\ref{thm:widder_representation} and is Bellman-side NDE on $X$. 
% If its complete-information boundary function is
% a polynomial of degree at most $n$, i.e., $\phi_J(m)=J(m,0)=\sum_{i=0}^{n}c_i m^i$,  then
% \[
%     J(m,P)=\sum_{i=0}^{n}c_i\,h_i(m,P),
%     \qquad h_i(m,P)=E\{(m+\sqrt P\,z)^i\}.
% \]
% Let Assumptions~\ref{ass:regularity}
% and~\ref{ass:boundary_trace} hold.
Let \(J\in\mathcal H_{\mathrm{NDE}}\).
Then \(J\) has a complete-information boundary function that is a polynomial
of degree at most \(n\), i.e., $\phi_J(m)=J(m,0)=\sum_{i=0}^{n}c_i m^i$, if and only if \(J\in\mathcal H_n\).
% If a Bellman-side NDE function $J$ on $X$ has a complete-information boundary function $\phi_J(m)$ that is a polynomial of degree at most $n$, i.e., $\phi_J(m)=J(m,0)=\sum_{i=0}^{n}c_i m^i$, then $J$ belongs to the finite-dimensional space
% \[
%     \mathcal H_n=\operatorname{span}\{h_0,\ldots,h_n\}.
% \]
% Conversely, every function in $\mathcal H_n$ is Bellman-side NDE on $X$ and has a
% complete-information boundary function $\phi_J(m)$  that is a polynomial of degree at most $n$.
\end{lemma}

\begin{proof}
Since $J$ is Bellman-side NDE on $X$, Theorem~\ref{thm:widder_representation}
gives
\( 
    J(m,P)=E\{\phi_J(m+\sqrt P\,z)\}.
\)
If $\phi_J(m)=\sum_{i=0}^{n}c_i m^i$, then by linearity of expectation,
% \[
%     J(m,P)
%     =E\Bigl\{\sum_{i=0}^{n}c_i(m+\sqrt P\,z)^i\Bigr\}
%     =\sum_{i=0}^{n}c_i\,E\{(m+\sqrt P\,z)^i\}
%     =\sum_{i=0}^{n}c_i\,h_i(m,P),
% \]
\[
\begin{aligned}
    J(m,P)
    &= E\Bigl\{\sum_{i=0}^{n}c_i(m+\sqrt P\,z)^i\Bigr\} \\
    &= \sum_{i=0}^{n}c_i\,E\{(m+\sqrt P\,z)^i\}
     = \sum_{i=0}^{n}c_i\,h_i(m,P).
\end{aligned}
\]
so $J\in\mathcal H_n$.

% Conversely, suppose $J=\sum_{i=0}^{n}c_i h_i$. Each $h_i(m,P)=E\{(m+\sqrt P\,z)^i\}$
% is the Gaussian convolution of the boundary function $m^i$, hence $h_i$ is Bellman-side NDE
% on $X$ by Theorem~\ref{thm:widder_representation}.
% By linearity, $J$ is Bellman-side NDE.
% Moreover, since $h_i(m,0)=m^i$, the complete-information boundary function of $J$
% is 
Conversely, suppose
\( 
    J=\sum_{i=0}^{n}c_i h_i
\).
Since \(h_i(m,0)=m^i\), the complete-information boundary function of \(J\) is
$\phi_J(m)=J(m,0)=\sum_{i=0}^{n}c_i m^i$, which is a polynomial of degree at most $n$. This proves
the equivalence.
\end{proof}

\section{Bellman-side Function Decomposition}
\label{sec:bellman_side_decomposition}

Lemma~\ref{lem:B_to_H} shows that the Gaussian convolution of a complete-information boundary function in \(\mathcal B\) is Bellman-side NDE.
This suggests measuring the dual effect of a general cost-to-go function by its deviation from this Gaussian convolution.

Fix a cost-to-go function $J\in \mathcal H$. 
 The associated \emph{Bellman-side NDE component} is the Gaussian convolution of the complete-information boundary function $\phi_J$, with $z\sim N(0,1)$,
 % defined in~\eqref{eq:phi-J},
\begin{equation}
    J_{\mathrm{NDE}}(m,P)=
    J_{\phi_J}(m,P)
    =
    E\{\phi_J(m+\sqrt{P}\,z)\},
    \label{eq:Jnde_def}
\end{equation}
and the remaining \emph{DE component} is defined by the residual
\begin{equation}
    J_{\mathrm{DE}}(m,P)
    =
    J(m,P)-J_{\mathrm{NDE}}(m,P).
    \label{eq:Jde_def}
\end{equation}
Since \(J\in\mathcal H\),
\(J_{\mathrm{NDE}}\in\mathcal H_{\mathrm{NDE}}\subset\mathcal H\),
and Assumptions~\ref{ass:regularity} and~\ref{ass:boundary_trace} are preserved under finite linear combinations,
we also have \(J_{\mathrm{DE}}\in\mathcal H\).
The following decomposition then holds by construction
\begin{equation}
    J(m,P)=J_{\mathrm{NDE}}(m,P)+J_{\mathrm{DE}}(m,P).
    \label{eq:dec}
\end{equation}
% then holds by construction, requiring no further argument; the content lies
% entirely in the properties of the two components.
%========================= (2) uniqueness =========================
This decomposition is unique, as stated in the following lemma.
\begin{lemma}[Uniqueness of the NDE--DE decomposition]
\label{lem:unique_decomposition}
Let \(J\in\mathcal H\) have complete-information boundary function
\(\phi_J\). Then there exists a unique decomposition~\eqref{eq:dec} such that
\(J_{\mathrm{NDE}}\in\mathcal H_{\mathrm{NDE}}\) has complete-information
boundary function \(\phi_J\). In particular,
\(J_{\mathrm{NDE}}\) and \(J_{\mathrm{DE}}\) are uniquely given
by~\eqref{eq:Jnde_def} and~\eqref{eq:Jde_def}, respectively.
\end{lemma}

\begin{proof}
By Lemma~\ref{lem:NDE_space}, there exists a unique
\(J_{\mathrm{NDE}}\in\mathcal H_{\mathrm{NDE}}\), i.e., the Gaussian convolution $J_{\phi_J}$
in~\eqref{eq:Jnde_def}, with complete-information
boundary function \(\phi_J\). Since \(J\) is fixed,
\(
    J_{\mathrm{DE}}\)
is then uniquely determined by~\eqref{eq:Jde_def}. Hence the decomposition
is unique.
\end{proof}

The next theorem records the three
properties of $J_{\mathrm{DE}}$ that make it the carrier of the dual effect.

%========================= (3) theorem: two legs =========================
\begin{theorem}[Three
properties of $J_{\mathrm{DE}}$]
\label{thm:source_gap}
Suppose $J$ satisfies Assumptions~\ref{ass:regularity}
and~\ref{ass:boundary_trace}. Then:
\begin{enumerate}
    \item[(a)] $J_{\mathrm{DE}}$ has zero boundary trace,
    \begin{equation}
        J_{\mathrm{DE}}(m,0)=0,
        \qquad\forall\, m\in\Re.
        \label{eq:Jde_zero_trace}
    \end{equation}
    \item[(b)] $J_{\mathrm{DE}}$ carries the full source of $J$,
    \begin{equation}
        \mathcal S_J(m,P)
        =
        \mathcal S_{J_{\mathrm{DE}}}(m,P),
        \qquad \forall x\in X.
        \label{eq:source_invariance}
    \end{equation}

     \item[(c)] $J_{\mathrm{DE}}$ carries the entire dual-effect function:
\[
    \Delta_J(x,\eta)
    =
    \Delta_{J_{\mathrm{DE}}}(x,\eta),
    \qquad
    \forall x\in X,\ \eta\in[0,P).
\]
\end{enumerate}
\end{theorem}

\begin{proof}
\emph{(a).} As the Gaussian convolution of $\phi_J$, the component
$J_{\mathrm{NDE}}$ attains $\phi_J$ as its boundary trace, so that $J_{\mathrm{NDE}}(m,0)=\phi_J(m)$.
% in the shorthand of~\eqref{eq:phi-J}. 
By~\eqref{eq:phi-J}, we have \ $J(m,0)=\phi_J(m)$. Subtracting these two equalities
\[
    J_{\mathrm{DE}}(m,0)
    =J(m,0)-J_{\mathrm{NDE}}(m,0)
    =0,
    \qquad\forall\, m\in\Re.
\]

\emph{(b).} The source operator $\mathcal S_J=\partial_P J-\tfrac12 \partial_{mm} J$ is linear in
$J$. Since $J_{\mathrm{NDE}}\in\mathcal H_{\mathrm{NDE}}$ solves the heat
equation, its source vanishes: $\mathcal S_{J_{\mathrm{NDE}}}\equiv0$ on $X$.
Hence, by the definition~\eqref{eq:Jde_def} of $J_{\mathrm{DE}}$ and the linearity of the source operator,
\[
    \mathcal S_{J_{\mathrm{DE}}}
    \equiv\mathcal S_{J}-\mathcal S_{J_{\mathrm{NDE}}}
    \equiv\mathcal S_{J}
    \qquad\text{on } X. 
\]

\emph{(c).} By Theorem~\ref{thm:exact_source_representation},
\( 
    \Delta_J(x,\eta)   =
    -
    \int_0^\eta
    E\!\left\{
        \mathcal S_J(m+\sqrt{s}w,P-s)
    \right\}ds
\). Thus, property~(b),
$    \mathcal S_{J_{\mathrm{DE}}}
    \equiv\mathcal S_{J}
$ on $X$, gives property~(c).
\end{proof}

%========================= (4) corollary: use both legs =========================

Theorem~\ref{thm:source_gap} attributes the entire dual effect of $J$ to the
single component $J_{\mathrm{DE}}$. 
The component $J_{\mathrm{NDE}}$ contributes a $u$-independent term
to the Bellman objective~\eqref{eq:Q}. For a fixed stage cost, CE optimality is thus decided by $J_{\mathrm{DE}}$ alone.
We next characterize Bellman-side NDE and CE optimality in the case
where $J_{\mathrm{DE}}$ vanishes.

\begin{corollary}[Bellman-side NDE and CE optimality on \(J_{\mathrm{DE}}\)]
\label{cor:nde_iff_Jde_zero}
Under the assumptions of Theorem~\ref{thm:source_gap}, \(J\) is Bellman-side
NDE on \(X\) if and only if
\begin{equation}
    J_{\mathrm{DE}}\equiv0
    \qquad\text{on }X.
    \label{eq:nde_iff}
\end{equation}
If this condition holds and Assumption~\ref{ass:ce_centered_cost} is satisfied, the CE input \(\nu(x)\) in~\eqref{eq:ce_input} uniquely minimizes the Bellman
objective \(Q_J(x,\cdot)\) in~\eqref{eq:Q} at every \(x\in X\).
\end{corollary}
\begin{proof}
% By definition,
% \(
%     J_{\mathrm{NDE}}
%     =
%     J_{\phi_J}\),
%     \(J_{\mathrm{DE}}
%     =
%     J-J_{\mathrm{NDE}}.
% \)
By Theorem~\ref{thm:widder_representation}, \(J\) is Bellman-side NDE on
\(X\) if and only if
\(
    J
    =
    J_{\phi_J}
    =
    J_{\mathrm{NDE}}\)
on $X$.
By the definition of \(J_{\mathrm{DE}}\), this is equivalent to
\(
    J_{\mathrm{DE}}\equiv 
    J-J_{\mathrm{NDE}}
    \equiv0
\) on $X$.
The final CE-optimality claim follows from
Theorem~\ref{thm:CE_opt_value_side}.
\end{proof}

% \begin{corollary}[Bellman-side NDE via $J_{\mathrm{DE}}$]
% \label{cor:nde_iff_Jde_zero}
% Under the assumptions of Theorem~\ref{thm:source_gap}, $J$ is Bellman-side NDE on
% $X$ if and only if
% \begin{equation}
%     J_{\mathrm{DE}}\equiv0
%     \qquad\text{on } X.
%     \label{eq:nde_iff}
% \end{equation}
% If, in addition, Assumption~\ref{ass:ce_centered_cost} holds, then
% in this case the CE input $\nu(x)$ in~\eqref{eq:ce_input} uniquely minimizes the Bellman objective $Q_J(x,\cdot)$ in~\eqref{eq:Q} at every $x\in X$.
% \end{corollary}

% \begin{proof}
% By Theorem~\ref{thm:source_gap}(c), $\Delta_J\equiv\Delta_{J_{\mathrm{DE}}}$ on
% $X$, so $J$ is Bellman-side NDE if and only if $J_{\mathrm{DE}}$ is. It therefore
% suffices to show that $J_{\mathrm{DE}}$ is Bellman-side NDE if and only if
% $J_{\mathrm{DE}}\equiv0$ on $X$.

% ($\Rightarrow$) Suppose $J_{\mathrm{DE}}$ is Bellman-side NDE. By Theorem~\ref{thm:widder_representation}, $J_{\mathrm{DE}}$ equals
% the Gaussian convolution of its boundary trace. This trace vanishes by
% Theorem~\ref{thm:source_gap}(a), so $J_{\mathrm{DE}}\equiv0$ on $X$.

% ($\Leftarrow$) Conversely, if $J_{\mathrm{DE}}\equiv0$, its dual-effect function
% is identically zero, so $J_{\mathrm{DE}}$ is Bellman-side NDE.

% The final CE-optimality claim follows from Theorem~\ref{thm:CE_opt_value_side}.
% \end{proof}

\begin{remark}
As noted earlier, Bellman-side NDE is sufficient, but not necessary, for CE
optimality. Corollary~\ref{cor:nde_iff_Jde_zero} characterizes Bellman-side NDE
by \(J_{\mathrm{DE}}\equiv0\), and hence provides one sufficient condition for
CE optimality. More generally, CE optimality may still hold when
\(J_{\mathrm{DE}}\not\equiv0\), provided that the contribution of
\(J_{\mathrm{DE}}\) to the expected cost-to-go does not shift the minimizer of
the Bellman objective away from the CE input. Thus, the structure of
\(J_{\mathrm{DE}}\) may yield less restrictive sufficient conditions for CE
optimality beyond Bellman-side NDE. Developing such conditions lies outside
the scope of the present paper.
\end{remark}

%The decomposition~\eqref{eq:dec} 
%thus isolates the dual effect in the single
%component $J_{\mathrm{DE}}$: since $J_{\mathrm{NDE}}$ is Bellman-side NDE, it
%contributes nothing to the dual-effect function, and the input affects the
%expected cost-to-go~\eqref{eq:expected_future_cost} only through $J_{\mathrm{DE}}$.
%Diagnosing and quantifying the dual effect of $J$ thus reduces to studying $J_{\mathrm{DE}}$.

\section{Numerical Examples}
\label{sec:numerical}

This section verifies the preceding characterizations numerically with some generic cost-to-go functions, which need not be generated by backward DP recursion.
% Section~\ref{subsec:quadratic}
%tests the source characterization of Section~\ref{sec:source} on the quadratic
%benchmark; 
%Section~\ref{subsec:quartic} tests the Gaussian-convolution
%representation of Section~\ref{sec:NDE_space} and the decomposition of
%Section~\ref{sec:decomposition} on the quartic benchmark.
Throughout, we take $k=1$ and $b=0$ in Examples~\ref{eg:quadratic}
and~\ref{eg:quartic}, so that the oracle input is $\upsilon^*(\theta)=\theta$ and
the CE input is $\nu(x)=m$. The quadratic stage cost~\eqref{eq:g_quadratic}
and the quartic stage cost~\eqref{eq:g_quartic} then reduce to
\begin{equation}\label{eq:g_quadratic_normalized}
	g(x,u) = (u-m)^2 + P,
\end{equation}
and
\begin{equation}\label{eq:g_quartic_normalized}
	g(x,u) = (u-m)^4 + 6P(u-m)^2 + 3P^2,
\end{equation}
respectively.

\subsection{Quadratic Cost Example}
\label{subsec:quadratic}

We first revisit the quadratic benchmark of Examples~\ref{eg:quadratic}
and~\ref{eg:quadratic_cont}. With $g(x,u)$ in~\eqref{eq:g_quadratic_normalized} and
$J(m,P)=m^2+\lambda P$, the Bellman objective satisfies
\begin{align}
\label{eq:policy_min}
    & \ Q_J(x,u) \notag\\
    &=(u-m)^2+P+E\big\{(m+\sqrt{\eta(P,u)}\,w)^2\!+\!\lambda\big(P\!-\!\eta(P,u)\big)\big\}
    \notag\\
    &=(u-m)^2+P+m^2+\eta(P,u)+\lambda\big(P-\eta(P,u)\big)
    \notag\\
    &=m^2+2P+(u-m)^2+(\lambda-1)\frac{P}{1+Pu^2},
\end{align}
where the second equality uses $w\sim N(0,1)$ to evaluate
$E\{(m+\sqrt{\eta}\,w)^2\}=m^2+\eta$, and the third uses
$P-\eta(P,u)=\frac{P}{1+Pu^2}$ from the variance reduction~\eqref{eq:eta}.
% The exploration-exploitation trade-off is shown in~\eqref{eq:policy_min}
% explicitly: The term $(u-m)^2$ penalizes deviations from the CE input; the last
% term captures how the posterior variance $\frac{P}{1+Pu^2}$ is valued. 

We numerically approximate the minimizers of the Bellman objective
in~\eqref{eq:policy_min} through an exhaustive search over a uniform
grid of $2001$ input values in $[-6,6]$.
At each belief state, we select the grid point with the lowest
objective value.
The minimizers are denoted by $\mu^*(m,P)$ for $\lambda\in\{0.5,1,2\}$. 
Fig.~\ref{fig:opt_policies} shows the resulting policy surfaces. The numerical 
results confirm the conclusion in Example~\ref{eg:quadratic_cont}: the sign of
$\lambda-1$ determines whether variance reduction, equivalently information acquisition, is rewarded, neutral, or discouraged.

\begin{itemize}
\item (Information neutrality.) When $\lambda=1$, the coefficient of $\frac{P}{1+Pu^2}$ is zero. Hence, the CE input $u=m$ is optimal, as shown by the orange surface in Fig.~\ref{fig:opt_policies}.

\item (Positive VOI.) 
When $\lambda>1$, the coefficient of $\frac{P}{1+Pu^2}$ is positive. A larger $|u|$ reduces this term and hence the expected cost-to-go. For $m\neq0$, the minimizer deviates from the CE input $u=m$ toward a larger $|u|$ to acquire information, as shown by the green surface
in Fig.~\ref{fig:opt_policies}.

\item (Negative VOI.)
When $\lambda<1$, the coefficient of $\frac{P}{1+Pu^2}$ is negative. A larger $|u|$ decreases $P/(1+Pu^2)$ and therefore increases the expected cost-to-go.
For $m\neq0$, the minimizer deviates from the CE input $u=m$ toward a smaller $|u|$, discouraging exploration, as shown by the blue surface in Fig.~\ref{fig:opt_policies}.
\end{itemize}

\begin{figure}
    \centering
    \includegraphics[width=0.8\linewidth]{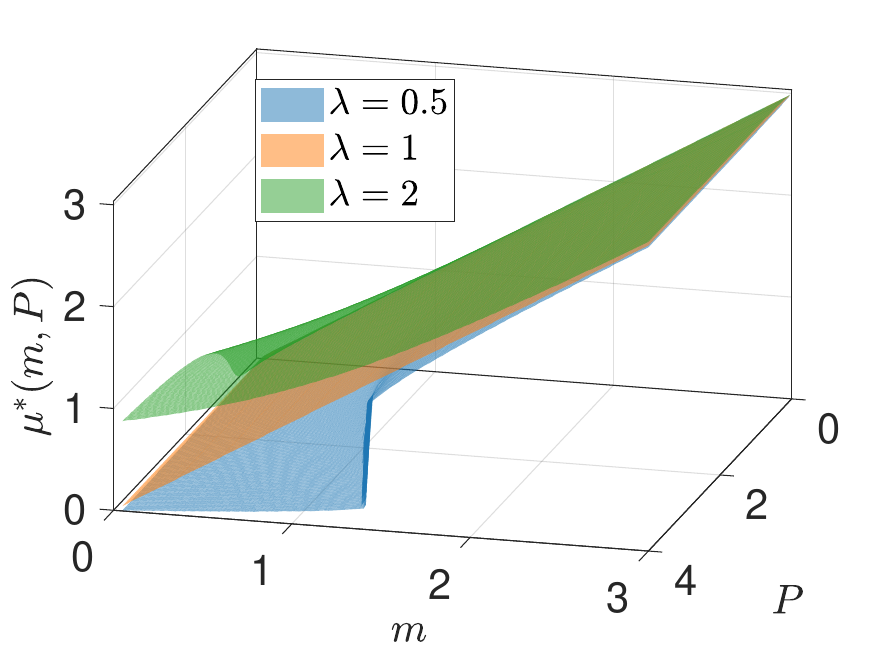}
    \caption{Minimizers of the Bellman objective~\eqref{eq:policy_min} in the  quadratic benchmark for
$\lambda=0.5,1,2$. The case $\lambda=1$ coincides with the CE policy, whereas
$\lambda>1$ favors exploration and $\lambda<1$ discourages exploration.}
    \label{fig:opt_policies}
\end{figure}

\subsection{Quartic Cost Example}
\label{subsec:quartic}

Consider the quartic stage cost given in~\eqref{eq:g_quartic_normalized}, which is minimized at the CE input $u=m$.

To construct $J$, we take the complete-information boundary function
$\phi_J(m)=m^4$. 
Thus, the Bellman-side NDE component $J_{\mathrm{NDE}}$ of $J$ is the Gaussian convolution of $\phi_J$:
\[
    J_{\mathrm{NDE}}(m,P)=E\{\phi_J(m+\sqrt P\,z)\}%=E\{(m+\sqrt P\,z)^4\}
    =m^4+6m^2P+3P^2.
\]
We then add a dual-effect component $J_{\mathrm{DE}}(m,P)=\beta P^2$, depending
only on the posterior variance, to obtain
\[
    J(m,P)\!=\!J_{\mathrm{NDE}}(m,P)\!+\!J_{\mathrm{DE}}(m,P)\!=\!m^4+6m^2P+(3+\beta)P^2.
    %\qquad \beta\in\Re.
\]
The source of $J$ is
$$
\begin{aligned}
    \mathcal S_J(m,P)
    % &
    % =
    % \partial_P J-\frac12 \partial_{mm} J 
    % \\
    % &
    =
    \mathcal S_{J_{\mathrm{NDE}}}(m,P)
    +
    \mathcal S_{J_{\mathrm{DE}}}(m,P) 
    % \\
    % &
    =
    0+2\beta P 
    % \\
    % &
    =
    2\beta P.
\end{aligned}
$$
According to Corollary~\ref{cor:zero_source}, $\beta>0$ corresponds to a positive-VOI $J$, $\beta=0$ corresponds to Bellman-side NDE, and $\beta<0$
corresponds to a negative-VOI $J$.

To verify this, we numerically evaluate the Bellman objective~\eqref{eq:Q}.
%\begin{equation}
%\label{eq:J_4rtic}
%    Q_J(x,u)=g(x,u)
%    +E\big\{J\big(m+\sqrt{\eta(P,u)}\,w,\,P-\eta(P,u)\big)\big\}.
%\end{equation}
Write $\eta=\eta(P,u)$. Since $J_{\mathrm{NDE}}$ is Bellman-side NDE, its expected cost-to-go is constant in $u$,
\[
    E\big\{J_{\mathrm{NDE}}\big(m+\sqrt{\eta}\,w,\,P-\eta\big)\big\}
    =J_{\mathrm{NDE}}(m,P).
\]
Since $J=J_{\mathrm{NDE}}+\beta P^2$, the expected cost-to-go~\eqref{eq:expected_future_cost} is
\begin{align}
    &\quad \ E\big\{J\big(m+\sqrt{\eta}\,w,\,P-\eta\big)\big\}\notag\\
    &=E\big\{J_{\mathrm{NDE}}\big(m+\sqrt{\eta}\,w,\,P-\eta\big)\big\}+\beta(P-\eta)^2
    \notag\\
    % &=J_{\mathrm{NDE}}(m,P)+\beta(P-\eta)^2
    % \notag\\
    &=J_{\mathrm{NDE}}(m,P)+\beta P^2-\beta\eta(2P-\eta),
\end{align}
where the second equality uses the Bellman-side NDE of $J_{\mathrm{NDE}}$ and expands $(P-\eta)^2$. Substituting into~\eqref{eq:Q} and using $g(x,u)$ in~\eqref{eq:g_quartic_normalized} yield
\begin{equation}
\label{eq:policy_min_4}
\begin{aligned}
    Q_J(x,u)
    &=(u-m)^4+6P(u-m)^2+3P^2 \\
    &\quad +J_{\mathrm{NDE}}(m,P)+\beta P^2
    -\beta\eta(2P-\eta) \\
    &=3P^2+J_{\mathrm{NDE}}(m,P)+\beta P^2 \\
    &\quad +\big[(u-m)^4+6P(u-m)^2-\beta\eta(2P-\eta)\big].
\end{aligned}
\end{equation}
where $\eta=\frac{P^2u^2}{1+Pu^2}$.
    
Fig.~\ref{fig:opt_policies_4} visualizes the policy $\mu^*(m,P)$ obtained by minimizing~\eqref{eq:policy_min_4} for $\beta\in\{-1,0,1\}$, obtained by applying a similar grid-search procedure to~\eqref{eq:policy_min}.
The results verify the source-based
conclusion: $\beta>0$ encourages exploration, $\beta=0$ recovers the CE policy, and $\beta<0$ discourages exploration.

\begin{figure}
    \centering
    \includegraphics[width=0.8\linewidth]{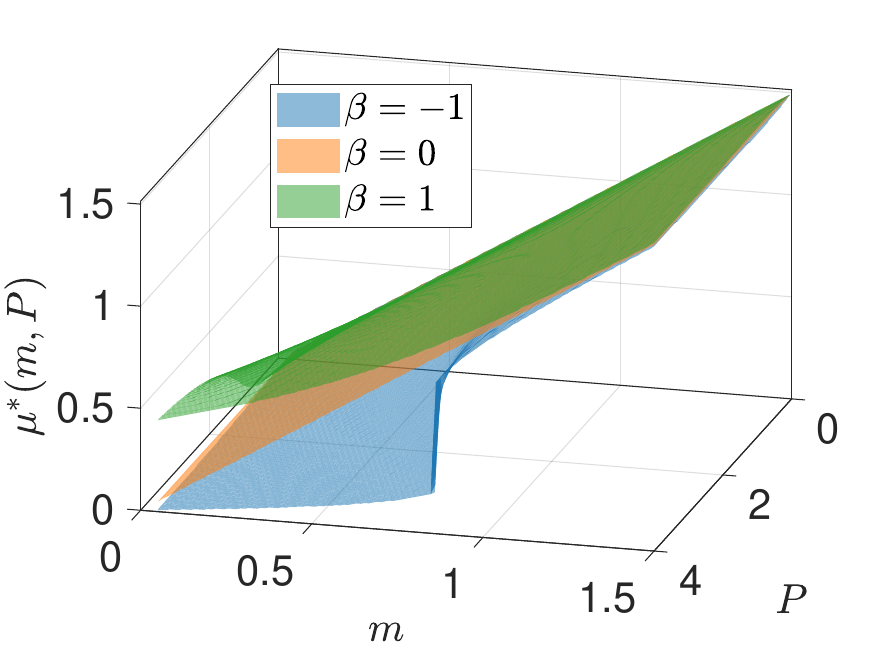}
    \caption{Minimizers of the Bellman objective~\eqref{eq:policy_min_4} in the quartic benchmark for
$\beta=-1,0,1$. The case $\beta=0$ coincides with the CE policy, whereas
$\beta=1$ favors exploration and $\beta=-1$ discourages exploration.}
    \label{fig:opt_policies_4}
\end{figure}

\section{Conclusion}
\label{sec:conclusion}

This paper revisited the relation between dual effect and CE optimality from a Bellman perspective. 
% Classical NDE notion requires that the conditional distribution of the estimation error is independent of the past control, and implies CE optimality.
Classical NDE, formulated on the posterior side, implies CE optimality but fails in the present setting.
% In our setting, this posterior-side condition does not hold. 
However, the Bellman objective minimization may still select the CE input when the expected cost-to-go is insensitive to the information generated by the input. 
This motivated us to introduce a notion of NDE at the level of the Bellman objective.

We characterized this Bellman-side NDE property, which is a sufficient condition for CE optimality. 
% We showed that it can be tested through a heat-equation diagnostic and.
We showed that for a given cost-to-go function, this condition can be checked by verifying whether the function satisfies a  heat equation.
We also established that Bellman-side NDE functions are Gaussian convolutions of complete-information boundary functions. 
% This representation identifies when a cost-to-go function has no Bellman-side dual effect, even though the posterior dynamics remain control-dependent. 
This representation also characterizes the space of Bellman-side NDE functions, and yields a unique decomposition of a cost-to-go function into a  no-dual-effect term and a dual-effect term.
Investigating this dual-effect term may yield less restrictive sufficient
conditions for CE optimality beyond Bellman-side NDE, which we leave for
future work.

% The framework separates two mechanisms that are often entangled in dual control problems. 
% Posterior-side dual effect describes how the control changes future information, whereas Bellman-side dual effect describes whether this change matters to the Bellman objective. 
% This distinction explains how CE optimality can hold beyond classical posterior-side NDE conditions and yields diagnostic tests, polynomial NDE bases, and decompositions into NDE and dual-effect components.

 Other future work includes extending the characterization to multi-input settings, where multiple exploration directions introduce additional geometric structure. Another possibility is to apply this work to concrete finite- and infinite-horizon DP problems, to identify conditions for policy-level CE optimality.

\appendices
\section{Proof of Lemma~\ref{lem:belief_state_transition_eta}}
\label{app:belief_state_transition_eta}

Conditional on $I_t$, we have $\theta-m_t\sim N(0,P_t)$, and $e_t\sim N(0,1)$ is independent of $\theta$ and $I_t$. Since $u_t$ is $I_t$-measurable,
\[
    y_t-u_tm_t=u_t(\theta-m_t)+e_t \!\sim\! N\bigl(0,\,1+P_tu_t^2\bigr)
    \ \text{conditional on } I_t.
\]
Hence $w_t=\operatorname{sgn}(u_t)(y_t-u_tm_t)/\sqrt{1+P_tu_t^2}$ satisfies
$w_t\mid I_t\sim N(0,1)$: for $u_t\neq0$ the sign factor preserves the standard
normal law by symmetry, and for $u_t=0$ the convention $\operatorname{sgn}(0)=1$
gives $w_t=e_t\sim N(0,1)$. Since $x_t=(m_t,P_t)$ is $I_t$-measurable, it follows
that $w_t\sim N(0,1)$ conditional on $(x_t,u_t)$.

By the definition of $w_t$, the innovation is
$y_t-u_tm_t=\operatorname{sgn}(u_t)\sqrt{1+P_tu_t^2}\,w_t$. Substituting this into
the mean update~\eqref{eq:theta_update_app} and using
$u_t\operatorname{sgn}(u_t)=|u_t|$ and
$\sqrt{\eta(P_t,u_t)}=\frac{P_t|u_t|}{\sqrt{1+P_tu_t^2}}$,
$$
\begin{aligned}
        m_{t+1}
    &=m_t+\frac{P_tu_t}{1+P_tu_t^2}\,\operatorname{sgn}(u_t)\sqrt{1+P_tu_t^2}\,w_t\\
    &=m_t+\frac{P_t|u_t|}{\sqrt{1+P_tu_t^2}}\,w_t
    =m_t+\sqrt{\eta(P_t,u_t)}\,w_t.
\end{aligned}
$$
% where the case $u_t=0$ holds trivially, as $\eta(P_t,0)=0$ makes the last term
% vanish.

The variance update~\eqref{eq:P_update_app} gives
$P_t-P_{t+1}=P_t^2u_t^2/(1+P_tu_t^2)=\eta(P_t,u_t)$, so
$P_{t+1}=P_t-\eta(P_t,u_t)$. %This establishes the representation~\eqref{eq:dynamics_f_def}.

\section{Proof of Lemma~\ref{lem:regularity_consequences}}
\label{app:regularity_proof}

We first prove the domination~\eqref{eq:domination}.
For $\eta\in[0,\bar\eta]$, the points
$(m+\sqrt{\eta}\,w,P-\eta)$ lie in
$\Re\times\mathcal I$, where
\[
    \mathcal I=[P-\bar\eta,P].
\]
Since $\bar\eta<P$, $\mathcal I$ is a compact subset of $(0,\infty)$.
Assumption~\ref{ass:regularity}(b) therefore yields constants
$c_{\mathcal I}>0$ and $q_{\mathcal I}\geq0$ such that
\[
\begin{aligned}
        \max\bigl\{|J|,|\partial_m J|,|\partial_{mm} J|,&|\partial_P J|\bigr\}
    (m+\sqrt{\eta}\,w,P-\eta)\\
    \leq
    &c_{\mathcal I}
    \bigl(1+|m+\sqrt{\eta}\,w|^{q_{\mathcal I}}\bigr).
\end{aligned}
\]
Since
\[
    |m+\sqrt{\eta}\,w|
    \leq
    |m|+\sqrt{\bar\eta}\,|w|,
    \qquad
    \eta\in[0,\bar\eta],
\]
the right-hand side is bounded uniformly in $\eta\in[0,\bar\eta]$ by
\begin{equation}
    G(w)
    =
    c_{\mathcal I}
    \Bigl(
        1+
        \bigl(|m|+\sqrt{\bar\eta}\,|w|\bigr)^{q_{\mathcal I}}
    \Bigr).
    \label{eq:Gw}
\end{equation}
The function $G$ has polynomial growth. Since $w\sim N(0,1)$ has finite
absolute moments of all orders,
\[
    E\{|w|^kG(w)\}<\infty,
    \qquad \forall\,k\geq0.
\]

\emph{(a).}
The claim follows directly from the domination~\eqref{eq:domination} and
$E\{G(w)\}<\infty$.

\emph{(b).}
Fix a compact interval $[\eta_1,\eta_2]\subset(0,\bar\eta)$.
By Assumption~3(a), for every fixed $w\in\Re$, the chain
rule gives
\begin{align*}
\frac{d}{d\eta}J(m+\sqrt{\eta}\,w,P-\eta)
&=
\frac{w}{2\sqrt{\eta}}\,
\partial_m J(m+\sqrt{\eta}\,w,P-\eta)
\\
&\quad
-\partial_P J(m+\sqrt{\eta}\,w,P-\eta).
\end{align*}
For every $\eta\in[\eta_1,\eta_2]$, the domination bound
in (26) therefore yields
\[
\left|
\frac{d}{d\eta}J(m+\sqrt{\eta}\,w,P-\eta)
\right|
\le
\left(1+\frac{|w|}{2\sqrt{\eta_1}}\right)G(w).
\]
The right-hand side is independent of $\eta$ and integrable,
since
\[
E\left\{
\left(1+\frac{|w|}{2\sqrt{\eta_1}}\right)G(w)
\right\}
=
E\{G(w)\}
+
\frac{E\{|w|G(w)\}}{2\sqrt{\eta_1}}
<\infty.
\]
Hence, differentiation is interchanged with expectation~\cite[Thm.~2.27(b)]{folland1999real},
so $F_x$ is differentiable on $(\eta_1,\eta_2)$ with $F'_x(\eta)$ in~\eqref{eq:diff_under_int}.
Since $[\eta_1,\eta_2]\subset(0,\bar\eta)$ is arbitrary,
this proves (b).

% \emph{(b).}
% Fix a compact interval $[\eta_1,\eta_2]\subset(0,\bar\eta)$. By
% Assumption~\ref{ass:regularity}(a), the chain rule gives
% \[
% \begin{aligned}
%         % F_x'(\eta)
%     % =
%     &\frac{d}{d\eta}
%     J(m+\sqrt{\eta}\,w,P-\eta)\\
%     &=
%     \frac{w}{2\sqrt{\eta}}
%     \partial_m J(m+\sqrt{\eta}\,w,P-\eta)
%     -
%     \partial_P J(m+\sqrt{\eta}\,w,P-\eta).
% \end{aligned}
% \]
% For $\eta\in[\eta_1,\eta_2]$, the domination~\eqref{eq:domination} yields
% \[
%     \left|
%         % \frac{\partial}{\partial\eta}
%         % J(m+\sqrt{\eta}\,w,P-\eta)
%         F_x'(\eta)
%     \right|
%     \leq
%     \left(
%         1+\frac{|w|}{2\sqrt{\eta_1}}
%     \right)G(w), 
% \] %e[Thm.~2.27(b)]
% and the right-hand side is integrable. Hence, differentiation and expectation
% on the left-hand side of $F_x'(\eta)$ can be interchanged
% by~\cite[Thm.~2.27(b)]{folland1999real}. Since
% $[\eta_1,\eta_2]\subset(0,\bar\eta)$ was arbitrary,~(b) follows.

\emph{(c).}
Fix $\eta\in(0,\bar\eta]$ and define
\[
    B(w)
    =
    \partial_m J(m+\sqrt{\eta}\,w,P-\eta).
\]
By Assumption~\ref{ass:regularity}(a), $B$ is continuously differentiable on
$\Re$, with
\[
    B'(w)
    =
    \sqrt{\eta}\,
    \partial_{mm} J(m+\sqrt{\eta}\,w,P-\eta).
\]
Moreover, the domination~\eqref{eq:domination} gives
\[
    |B(w)|\leq G(w),
    \qquad
    |B'(w)|\leq\sqrt{\eta}\,G(w).
\]
Hence, the integrability conditions required for Stein's identity are satisfied:
\[
    E\{|wB(w)|\}
    \leq
    E\{|w|G(w)\}
    <\infty,
\]
and
\[
   E\{|B'(w)|\}
    \leq
    \sqrt{\eta}\,E\{G(w)\}
    <\infty.
\]
Thus Stein's identity applies and yields~\eqref{eq:stein_id}.

\emph{(d).}
Fix $\eta_0\in[0,\bar\eta]$ and let $\eta\to\eta_0$. For every fixed
$w\in\Re$,
\[
    (m+\sqrt{\eta}\,w,P-\eta)
    \to
    (m+\sqrt{\eta_0}\,w,P-\eta_0).
\]
By Assumption~\ref{ass:regularity}(a), $J$ is continuous on $X$, and hence
\[
    J(m+\sqrt{\eta}\,w,P-\eta)
    \to
    J(m+\sqrt{\eta_0}\,w,P-\eta_0).
\]
Moreover, by~\eqref{eq:domination},
\[
    |J(m+\sqrt{\eta}\,w,P-\eta)|
    \leq G(w),
    \qquad
    \eta\in[0,\bar\eta],
\]
where $E\{G(w)\}<\infty$. The dominated convergence theorem~\cite{rudin1974real} thus gives
\[
    F_x(\eta)
    \to
    F_x(\eta_0).
\]
Since $\eta_0\in[0,\bar\eta]$ was arbitrary, $F_x$ is continuous on
$[0,\bar\eta]$. In particular,
\[
    F_x(0)
    =
    E\{J(m,P)\}
    =
    J(m,P).
\]

\emph{(e).}
Combining~\eqref{eq:diff_under_int} with Stein's identity
in~\eqref{eq:stein_id} gives, for every $\eta\in(0,\bar\eta)$,
% \[
% \begin{aligned}
%     F_x'(\eta)
%     &=
%     E\left\{
%         \frac{1}{2}
%         \partial_{mm} J(m+\sqrt{\eta}\,w,P-\eta)
%         -
%         \partial_P J(m+\sqrt{\eta}\,w,P-\eta)
%     \right\}.
% \end{aligned}
% \]
% Therefore, by~\eqref{eq:domination},
% \[
% \begin{aligned}
%     |F_x'(\eta)|
%     &\leq
%     E\left\{
%         \frac{1}{2}
%         \big|
%             \partial_{mm} J(m+\sqrt{\eta}\,w,P-\eta)
%         \big|
%         +
%         \big|
%             \partial_P J(m+\sqrt{\eta}\,w,P-\eta)
%         \big|
%     \right\} \\
%     &\leq
%     \frac{3}{2}E\{G(w)\},
%     \qquad
%     \forall\,\eta\in(0,\bar\eta).
% \end{aligned}
% \]
\[
\begin{aligned}
    F_x'(\eta)
    &= E\Bigl\{
        \frac{1}{2}
        \partial_{mm} J(m+\sqrt{\eta}\,w,P-\eta) \\
    &\qquad {}-
        \partial_P J(m+\sqrt{\eta}\,w,P-\eta)
    \Bigr\}.
\end{aligned}
\]
Therefore, by~\eqref{eq:domination},
\[
\begin{aligned}
    |F_x'(\eta)|
    &\leq E\Bigl\{
        \frac{1}{2}
        \big|\partial_{mm} J(m+\sqrt{\eta}\,w,P-\eta)\big| \\
    &\qquad {}+
        \big|\partial_P J(m+\sqrt{\eta}\,w,P-\eta)\big|
    \Bigr\} \\
    &\leq \frac{3}{2}E\{G(w)\},
    \qquad \forall\,\eta\in(0,\bar\eta).
\end{aligned}
\]
Thus $F_x'$ is uniformly bounded on $(0,\bar\eta)$. By the mean value theorem,
for any $\eta_1,\eta_2\in(0,\bar\eta)$,
\[
    |F_x(\eta_2)-F_x(\eta_1)|
    \leq
    \frac{3}{2}E\{G(w)\}
    |\eta_2-\eta_1|.
\]
Using the continuity of $F_x$ on $[0,\bar\eta]$ established in~(d), the same
bound extends to $\eta_1,\eta_2\in[0,\bar\eta]$. Hence $F_x$ is Lipschitz
continuous on $[0,\bar\eta]$, and therefore absolutely continuous on
$[0,\bar\eta]$.

\section{Proof of Lemma~\ref{lem:local_voi}}
\label{app:local_voi}
% \begin{proof}
By Theorem~\ref{thm:exact_source_representation},
\begin{equation}
        \Delta_J(x,\eta)
% =E\big\{
%         J\big(m+\sqrt{\eta}\,w,\ P-\eta\big)
%     \big\}
%     -
%     J(m,P)
    =
    -
    \int_0^\eta
    E\big\{
        \mathcal S_J\big(m+\sqrt{s}\,w,\ P-s\big)
    \big\}\,ds .\label{eq:delta_J_int}
\end{equation}
It remains to characterize the behavior of the integrand as \(s\downarrow0\).
Define
\[
    H(s)
    =
    E\big\{
        \mathcal S_J\big(m+\sqrt{s}\,w,\ P-s\big)
    \big\}.
\]

First, we verify pointwise convergence  of the integrand. Since \(\partial_P J\)
and \(\partial_{mm} J\) are continuous on \(X\) by
Assumption~\ref{ass:regularity}(a), the source \(\mathcal S_J\) is
continuous on \(X\). Hence, for every fixed \(w\in\Re\),
\begin{equation}
\label{eq:dct_cont}
    \mathcal S_J\big(m+\sqrt{s}\,w,\ P-s\big)
    \to
    \mathcal S_J(m,P),
    \qquad
    s\downarrow0 .
\end{equation}

Second, we establish an integrable dominating function. Applying
Lemma~\ref{lem:regularity_consequences} with $\bar\eta=P/2$ in
\eqref{eq:domination} gives, for every $s\in[0,P/2]$,
\[
    |\partial_P J(m+\sqrt{s}\,w,P-s)|
    \leq G(w),
\]
and
\[
    |\partial_{mm} J(m+\sqrt{s}\,w,P-s)|
    \leq G(w).
\]
Thus,
% \begin{equation}
% \label{eq:dct_bound}
%     \big|
%         \mathcal S_J\big(m+\sqrt{s}\,w,\ P-s\big)
%     \big|
%     \le
%     \big|
%          \partial_P J\big(m+\sqrt{s}\,w,\ P-s\big)
%     \big|
%     +
%     \frac{1}{2}\big|
%          \partial_{mm} J\big(m+\sqrt{s}\,w,\ P-s\big)
%     \big|
%     \leq
%     \frac{3}{2}G(w).
% \end{equation}
\begin{equation}
\label{eq:dct_bound}
\begin{aligned}
    \big|\mathcal S_J(m+\sqrt{s}\,w,P-s)\big| \le
    &\big|\partial_P J(m+\sqrt{s}\,w,P-s)\big| \\
    &+
    \frac{1}{2}\big|\partial_{mm} J(m+\sqrt{s}\,w,P-s)\big| \\
    \le &\frac{3}{2}G(w).
\end{aligned}
\end{equation}
Since $E\{G(w)\}<\infty$, the dominated convergence theorem~\cite{rudin1974real} gives
\[
    % H(s)
    % =
    E\big\{
        \mathcal S_J\big(m+\sqrt{s}\,w,\ P-s\big)
    \big\}
    \to
    \mathcal S_J(m,P),
    \qquad
    s\downarrow0 .
\]
Equivalently,
\[
        E\big\{
        \mathcal S_J\big(m+\sqrt{s}\,w,\ P-s\big)
    \big\}
    =
    \mathcal S_J(m,P)+o(1),
    \qquad
    s\downarrow0.
\]
Substituting this expansion into~\eqref{eq:delta_J_int} proves \eqref{eq:local_expansion}, i.e.,
\[
\begin{aligned}
    \Delta_J(x,\eta)
    % &=
    % -\int_0^\eta     E\big\{
    %     \mathcal S_J\big(m+\sqrt{s}\,w,\ P-s\big)
    % \big\}
    % \,ds \\
    \!=\!
    -\int_0^\eta
    \bigl(\mathcal S_J(m,P)+o(1)\bigr)\,ds 
    \!=\!
    -\eta\,\mathcal S_J(m,P)+o(\eta),
\end{aligned}
\]
as $ \eta\downarrow0$.  Consequently,
\[
    \frac{\Delta_J(x,\eta)}{\eta}
    \to
    -\mathcal S_J(m,P),
    \qquad
    \eta\downarrow0.
\]
If $\mathcal S_J(m,P)=0$, then
$\Delta_J(x,\eta)=o(\eta)$, which proves~(a).
If $\mathcal S_J(m,P)>0$, then the above ratio is negative for all
sufficiently small $\eta>0$, and hence $\Delta_J(x,\eta)<0$, proving~(b).
Similarly, if $\mathcal S_J(m,P)<0$, then the ratio is positive for all
sufficiently small $\eta>0$, and hence $\Delta_J(x,\eta)>0$, proving~(c).

\section{Proof of Lemma~\ref{lem:H_to_B}}
\label{app:H_to_B}
	We first prove continuity. By Assumption~\ref{ass:regularity}(a), $J(\cdot,P)$ is
	continuous on $\Re$ for each $P>0$. By Assumption~\ref{ass:boundary_trace}(a),
	$J(\cdot,P)\to\phi_J$ as $P\downarrow0$ uniformly on every compact interval
	$K\subset\Re$. A locally uniform limit of continuous functions is continuous~\cite[Thm.~7.12]{rudin1976principles},
	so $\phi_J$ is continuous on $\Re$.
	
	Then, we prove the polynomial growth. Fix $\bar P>0$ and let $c>0$, $q\ge0$ be the
	constants from Assumption~\ref{ass:boundary_trace}(b), so that
	\[
	|J(m,P)|\le c\,(1+|m|^{q}),
	\qquad \forall m\in\Re,\ P\in(0,\bar P].
	\]
	Fix $m\in\Re$. Letting $P\downarrow0$ and using
	Assumption~\ref{ass:boundary_trace}(a), $J(m,P)\to\phi_J(m)$, so 
	\[
	|\phi_J(m)|
	= \lim_{P\downarrow0}|J(m,P)|
	\le c\,(1+|m|^{q}).
	\]
	The bound holds for all $m\in\Re$ with the same  $c,q$, which is
	the claim.

\section{Proof of Lemma~\ref{lem:B_to_H}}
\label{app:B_to_H}

Since \(\phi\in\mathcal B\) is continuous and of polynomial growth, standard
properties of the Gaussian heat kernel imply that
\(J_\phi\in C^\infty(X)\),
\[
    \partial_PJ_\phi
    =
    \frac12\partial_{mm}J_\phi,
\]
and
\[
    J_\phi(\cdot,P)\to\phi
    \qquad\text{locally uniformly as }P\downarrow0;
\]
see~\cite[Sec.~3.4, Thm.~3.41]{arendt2023partial} and
\cite[Ch.~4, Sec.~A]{folland1995partial}.
Hence \(\mathcal S_{J_\phi}\equiv0\) on \(X\), \(\phi_{J_\phi}=\phi\), and 
Assumptions~\ref{ass:regularity}(a)
and~\ref{ass:boundary_trace}(a) hold.

It remains to verify the required polynomial-growth bounds in Assumptions~\ref{ass:regularity}(b) and~\ref{ass:boundary_trace}(b). Since
\(\phi\in\mathcal B\), there exist \(c_\phi>0\) and \(q_\phi\ge0\) such that
\(|\phi(y)|\le c_\phi(1+|y|^{q_\phi})\) for all \(y\in\Re\). Using
\((a+b)^{q_\phi}\le2^{q_\phi}(a^{q_\phi}+b^{q_\phi})\) for \(a,b\ge0\), we
obtain the pointwise bound
\begin{equation}
\label{eq:gc_pointwise}
    |\phi(m+\sqrt P\,z)|
    \le
    c_\phi 2^{q_\phi}
    \left(
        1+|m|^{q_\phi}+P^{q_\phi/2}|z|^{q_\phi}
    \right),
\end{equation}
where $m,z\in\Re$,$ P>0 $.
Fix \(\bar P>0\). Taking expectations in \eqref{eq:gc_pointwise} over
\(z\sim N(0,1)\) and using the finiteness of Gaussian moments gives
\[
    |J_\phi(m,P)|
    \le
    E\big\{|\phi(m+\sqrt P\,z)|\big\}
    \le
    C_1(1+|m|^{q_\phi}),
\]
where $m\in\Re $, $ 0<P\le\bar P$, 
and \(C_1\) depends only on \(c_\phi\), \(q_\phi\) and \(\bar P\). Since
\(\bar P>0\) was arbitrary, this is Assumption~\ref{ass:boundary_trace}(b).

To bound the partial derivatives required in
Assumption~\ref{ass:regularity}(b), write the Gaussian convolution as
\[
    J_\phi(m,P)
    \!=\!
    \int_{\Re}
        \phi(y)K_P(m-y)\,dy,
    \
    K_P(r)
    \!=\!
    \frac{1}{\sqrt{2\pi P}}
    e^{-r^2/(2P)}.
\]
% The general differentiation rule for convolutions is obtained by
% differentiating the kernel under the integral sign; see, e.g.,
% \cite[Thm.~0.14]{folland1995partial}. 
Since $\phi$ has polynomial growth and the Gaussian kernel
derivatives decay faster than any polynomial, the integrands
$\phi(y)\partial_m K_P(m-y)$ and
$\phi(y)\partial_{mm} K_P(m-y)$ admit a common integrable
bound, locally uniformly in $(m,P)\in X$.
Hence, differentiation under the integral sign is justified.
Indeed,
\[
\begin{aligned}
    \partial_m K_P(m-y)
    &=
    -\frac{m-y}{P}K_P(m-y),\\
    \partial_{mm}K_P(m-y)
    &=
    \left(
        \frac{(m-y)^2}{P^2}-\frac1P
    \right)
    K_P(m-y).
\end{aligned}
\]
Thus, each derivative is a polynomial multiplying the
Gaussian kernel, and hence remains integrable against a polynomially growing
\(\phi\) for \(P\) bounded away from zero. Therefore,
\[
\begin{aligned}
    \partial_mJ_\phi(m,P)
    &=
    \int_{\Re}
        \phi(y)\partial_mK_P(m-y)\,dy,\\
    \partial_{mm}J_\phi(m,P)
    &=
    \int_{\Re}
        \phi(y)\partial_{mm}K_P(m-y)\,dy.
\end{aligned}
\]
Using the change of variables \(y=m+\sqrt P\,z\), with \(z\sim N(0,1)\),
these expressions become
\[
\begin{aligned}
    \partial_mJ_\phi(m,P)
    &=
    P^{-1/2}
    E\big\{
        z\,\phi(m+\sqrt P\,z)
    \big\},\\
    \partial_{mm}J_\phi(m,P)
    &=
    P^{-1}
    E\big\{
        (z^2-1)\phi(m+\sqrt P\,z)
    \big\}.
\end{aligned}
\]
Now let \(I=[P_1,P_2]\subset(0,\infty)\) be compact. Since \(P\ge P_1>0\), and
using \(|z^2-1|\le1+z^2\),
\[
\begin{aligned}
    |\partial_mJ_\phi(m,P)|
    &\le
    P_1^{-1/2}
    E\big\{
        |\phi(m+\sqrt P\,z)|\,|z|
    \big\},\\
    |\partial_{mm}J_\phi(m,P)|
    &\le
    P_1^{-1}
    E\big\{
        |\phi(m+\sqrt P\,z)|(1+|z|^2)
    \big\}.
\end{aligned}
\]
Applying \eqref{eq:gc_pointwise} with \(P\le P_2\) and using the finiteness of
Gaussian moments, it follows that
\[
    |\partial_mJ_\phi(m,P)|
    +
    |\partial_{mm}J_\phi(m,P)|
    \le
    C_2(1+|m|^{q_\phi}),
\]
where $m\in\Re$, $ P\in I$, 
and \(C_2\) depends only on \(c_\phi\), \(q_\phi\), \(P_1\) and \(P_2\).
Together with
\(
    \partial_PJ_\phi
    =
    \frac12\partial_{mm}J_\phi
\)
and the bound for \(|J_\phi|\) established above with \(\bar P=P_2\), this
yields
\[
    \max\{
        |J_\phi|,
        |\partial_mJ_\phi|,
        |\partial_{mm}J_\phi|,
        |\partial_PJ_\phi|
    \}(m,P)
    \le
    c_I(1+|m|^{q_\phi}),
\]
where $ m\in\Re $, $P\in I$ and \(c_I\) is the largest of the constants involved. Thus Assumption~3(b)
holds with \(q_I=q_\phi\).

Therefore \(J_\phi\in\mathcal H\), and since \(\mathcal S_{J_\phi}\equiv0\)
on \(X\), the definition of \(\mathcal H_{\mathrm{NDE}}\) gives
\(J_\phi\in\mathcal H_{\mathrm{NDE}}\).

\section{Proof of Theorem~\ref{thm:widder_representation}}
\label{app:widder_representation}

Suppose first that \(J\) is Bellman-side NDE on \(X\). Then, for every
\(x=(m,P)\in X\) and every \(\eta\in[0,P)\),
\[
    J(m,P)
    =
    E\big\{
        J(m+\sqrt{\eta}\,z,P-\eta)
    \big\},
    \qquad z\sim N(0,1).
\]
As \(\eta\uparrow P\), for every fixed \(z\in\Re\),
\[
    J(m+\sqrt{\eta}\,z,P-\eta)
    \to
    \phi_J(m+\sqrt P\,z),
\]
by Assumption~\ref{ass:boundary_trace}(a) and the continuity of
\(\phi_J\). Moreover, Assumption~\ref{ass:boundary_trace}(b) provides an
integrable polynomial bound, uniform for \(\eta\) sufficiently close to
\(P\). Hence, by the dominated convergence theorem~\cite{rudin1974real},
\[
    J(m,P)
    =
    E\big\{
        \phi_J(m+\sqrt P\,z)
    \big\}
    =
    J_{\phi_J}(m,P).
\]
Thus \(J=J_{\phi_J}\) on \(X\).

Conversely, suppose that \(J=J_{\phi_J}\) on \(X\). Since
\(J\in\mathcal H\), its complete-information boundary function satisfies
\(\phi_J\in\mathcal B\). By Lemma~\ref{lem:B_to_H},
\(J_{\phi_J}\) is Bellman-side NDE on \(X\). Therefore \(J\) is
Bellman-side NDE on \(X\).

% \section*{Acknowledgment}

% The preferred spelling of the word ``acknowledgment'' in American English is 
% without an ``e'' after the ``g.'' Use the singular heading even if you have 
% many acknowledgments. Avoid expressions such as ``One of us (S.B.A.) would 
% like to thank $\ldots$ .'' Instead, write ``F. A. Author thanks $\ldots$ .'' In most 
% cases, sponsor and financial support acknowledgments are placed in the 
% unnumbered footnote on the first page, not here.

\section*{References}

\bibliographystyle{IEEEtran}
\bibliography{references}

\begin{IEEEbiography}[{\includegraphics[width=1in,height=1.25in,clip,keepaspectratio]{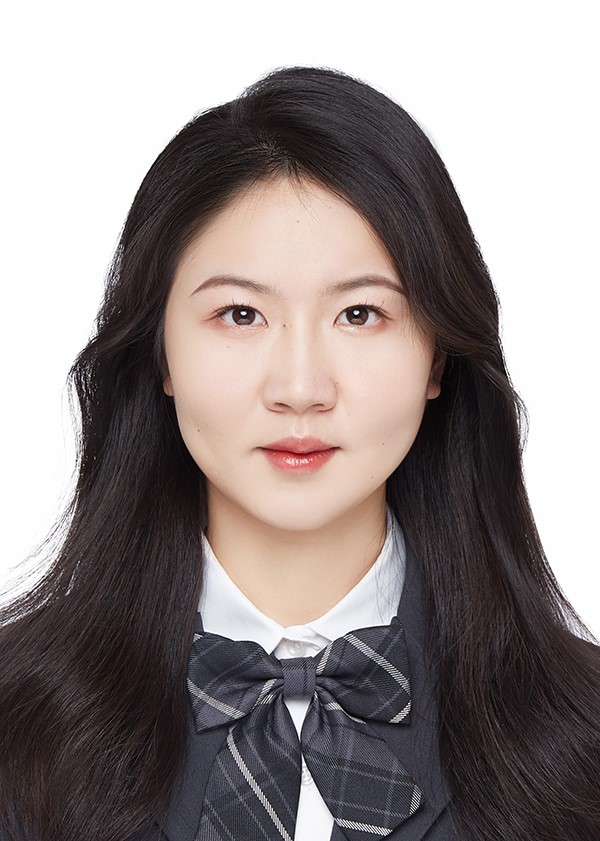}}] {Ying Wang} received the B.S. degree in automation from Northwestern Polytechnical University, Xi'an, China in 2020, and the M.Sc. degree in control science and engineering from Shanghai Jiao Tong University, Shanghai, China, in 2023.
She is currently pursuing a Ph.D. degree in electrical engineering at KTH Royal Institute of Technology, Stockholm, Sweden. Her research interests include dual control, system identification and input design.
\end{IEEEbiography}

\begin{IEEEbiography}[{\includegraphics[width=1in,height=1.25in,clip,keepaspectratio]{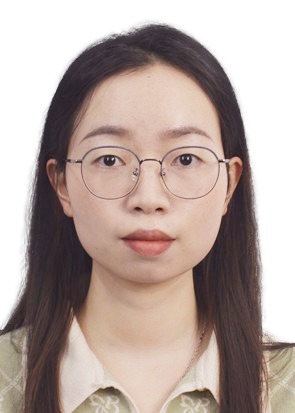}}] {Yue Ju} received the B.E. degree in automation from the Nanjing University of Science and Technology, Nanjing, China, in 2017 and the Ph.D. degree in computer and information engineering from the Chinese University of Hong Kong, Shenzhen, China, in 2022.
She is currently a postdoc at the KTH Royal Institute of Technology. She has been mainly working in the area of system identification.
\end{IEEEbiography}

\begin{IEEEbiography}[{\includegraphics[width=1in,height=1.25in,clip,keepaspectratio]{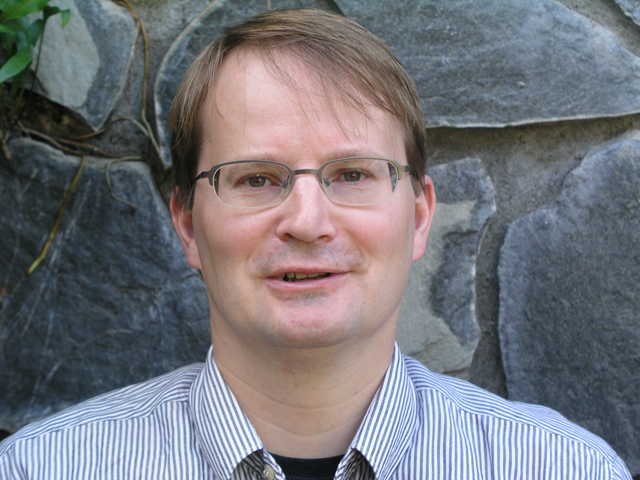}}] {H\r{a}kan Hjalmarsson} was born in 1962. He received the M.S. degree in Electrical Engineering in 1988, and the Licentiate degree and the Ph.D. degree in Automatic Control in 1990 and 1993, respectively, all from Link{\"o}ping University, Sweden. He has held visiting research positions at California Institute of Technology, Louvain University and at the University of Newcastle, Australia. He has served as an Associate Editor for Automatica (1996-2001), and IEEE Transactions on Automatic Control (2005-2007) and been Guest Editor for European Journal of Control and Control Engineering Practice. He is Professor at the Division of Decision and Control Systems, School of Electrical Engineering and Computer Science, KTH, Stockholm, Sweden and also affiliated with the Competence Centre for Advanced BioProduction by Continuous Processing, AdBIOPRO. He is an IEEE Fellow and past Chair of the IFAC Coordinating Committee CC1 Systems and Signals. In 2001, he received the KTH award for outstanding contribution to undergraduate education. He was General Chair for the IFAC Symposium on System Identification held in 2018. His research interests include system identification, learning of dynamical systems for control, process modeling control and also estimation in communication networks. 
\end{IEEEbiography}

\end{document}